\documentclass[12pt, reqno]{amsart}
\usepackage{amsmath, amstext, amsbsy, amssymb, amscd,mathrsfs}
\usepackage{amsmath}
\usepackage{amsxtra}
\usepackage{amscd}
\usepackage{amsthm}
\usepackage{amsfonts}
\usepackage{amssymb}
\usepackage{eucal}
\usepackage{color}

\input xy
\xyoption{all}

\input prepictex
\input postpictex

\newtheorem{theorem}{Theorem}[section]
\newtheorem*{thma}{Theorem A}
\newtheorem*{thmb}{Theorem B}
\newtheorem*{thmc}{Theorem C}
\newtheorem*{thmd}{Theorem D}
\newtheorem{lemma}[theorem]{Lemma}
\newtheorem{proposition}[theorem]{Proposition}
\newtheorem{corollary}[theorem]{Corollary}
\theoremstyle{definition}

\newtheorem{example}[theorem]{Example}
\newtheorem{remark}[theorem]{Remark}
\newtheorem{conjecture}[theorem]{Conjecture}

\numberwithin{equation}{section}

\newcommand{\g}{\mathfrak{g}}
\newcommand{\n}{\mathfrak{n}}
\newcommand{\h}{\mathfrak{h}}
\newcommand{\p}{\mathfrak{p}}

\newcommand{\m}{\mathfrak{m}}

\newcommand{\rea}[2]{U_{#1}(#2)} 

\newcommand{\pth}[1]{{#1}^{[p]}} 
\newcommand{\qn}{\mathfrak{q}(n)}
\newcommand{\ra}{\rightarrow}
\newcommand{\sudim}{\underline{\text{dim}}\,}
\newcommand{\ev}[1]{{#1}_{\bar{0}}}
\newcommand{\od}[1]{{#1}_{\bar{1}}}

\newcommand{\la}{\lambda}

\newcommand{\ad}{\text{ad}\,}

\newcommand{\Irr}{\text{Irr}}

\newcommand{\gl}{{\mathfrak{gl}}}

\newcommand{\Z}{ \mathbb Z }

\newcommand{\nc}{\newcommand}
 \nc{\La}{\Lambda}
 \nc{\ep}{\epsilon}
 \nc{\vep}{\varepsilon}

{\vskip-\lastskip\medskip
  \noindent
  {\em #1.}\enspace
  }%
{\qed\par\medskip
  }

\makeatletter
\def\Ddots{\mathinner{\mkern1mu\raise\p@
\vbox{\kern7\p@\hbox{.}}\mkern2mu
\raise4\p@\hbox{.}\mkern2mu\raise7\p@\hbox{.}\mkern1mu}}
\makeatother

\begin{document}
\title[Modular Representations of queer Lie superalgebras]
{Representations of Lie superalgebras in prime characteristic II:
the queer series}

\author[Weiqiang Wang]{Weiqiang Wang}\thanks{Partially supported by
NSA Grant H98230-08-1-0039 and NSF Grant DMS--0800280.}
\address{Department of Mathematics, University of Virginia,
Charlottesville, VA 22904} \email{ww9c@virginia.edu (Wang)\\
lz4u@virginia.edu (Zhao)}

\author{Lei Zhao}


\begin{abstract}
The modular representation theory of the queer Lie superalgebra
$\mathfrak{q}(n)$ over characteristic $p>2$ is developed. We
obtain a criterion for the irreducibility of baby Verma modules
with semisimple $p$-characters $\chi$ and a criterion for the
semisimplicity of the corresponding reduced enveloping algebras
$\rea{\chi}{\mathfrak{q}(n)}$. A $(2p)$-power divisibility of
dimensions of $\qn$-modules with nilpotent $p$-characters is
established. The representation theory of $\mathfrak{q}(2)$
is treated in detail. We formulate a Morita super-equivalence conjecture
for $\qn$ with general $p$-characters which is verified for $n=2$.
\end{abstract}

\maketitle
\date{}
  \setcounter{tocdepth}{1}
  \tableofcontents

\section{Introduction}\label{sec:intro}

\subsection{}

In \cite{WZ}, the authors initiated the modular representation
theory of Lie superalgebras over  an algebraically closed field
$K$ of characteristic $p>2$, by formulating a general superalgebra
analogue of the Kac-Weisfeiler (KW) conjecture and establishing it
for the basic classical Lie superalgebras. Our work generalized (via
a somewhat modified approach)
the earlier work on Lie algebras of reductive algebraic groups by
Kac-Weisfeiler \cite{KW}, Parshall-Friedlander \cite{FP}, Premet
\cite{Pr1, Pr2}, and others (cf. Jantzen \cite{Jan} for a review
and extensive references).

This paper is a sequel to \cite{WZ}, and its goal is to develop
systematically the modular representation theory of the queer Lie
superalgebra $\g \equiv \mathfrak{q}(n)$ over the field $K$. As a
byproduct, the finite $W$-algebra associated to $\g$ is also
introduced.

\subsection{}
Recall that the queer Lie superalgebra $\g =\ev \g +\od \g$ consists
of matrices of the form:
\begin{equation}  \label{q(n)}
\begin{pmatrix}
A&B\\
B&A\\
\end{pmatrix},
\end{equation}
where $A$ and $B$ are arbitrary $n\times n$ matrices. Note that the
even subalgebra $\ev \g$ is isomorphic to $\gl(n)$ and the odd part
$\od \g$ is another isomorphic copy of $\gl(n)$ under the adjoint
action of $\ev \g$. The queer Lie superalgebra $\g$ can be regarded
as a true super analogue of the general linear Lie algebra, and its
representation theory over the complex field has been studied by
various authors (see \cite{Br, CW, Pen, Ser, FM}). Also the
modular representations of the type $Q$ algebraic supergroup have
been studied in \cite{BK} and they played a key role in the the
classification of the simple modules of the spin symmetric group
over $K$.

It is worth emphasizing that, in contrast to simple Lie
algebras, a Cartan subalgebra $\h =\ev \h + \od \h$ of the queer
Lie superalgebra is not abelian and its odd part $\od \h$ is
nonzero. Moreover, $\qn$ admits a non-degenerate {\em odd}
symmetric bilinear form.

For restricted Lie superalgebras including $\g$, one can make
sense the notions of $p$-characters $\chi \in \ev \g^*$ and the
corresponding reduced enveloping algebras $U_\chi (\g)$ (see
Section~\ref{sec:basics}). Recalling $\ev \g \cong \gl (n)$, one
can also make sense the Jordan decomposition of a $p$-character as
well as the notion of semisimple and nilpotent $p$-characters. The
Lie superalgebra $\g$ admits a triangular decomposition $ \g
=\mathfrak{n}^- \oplus \mathfrak{h} \oplus \mathfrak{n}^+. $ We
may assume that a $p$-character $\chi$ satisfies $\chi(\ev \n^+)
=0$ without loss of generality, via a $GL(n)$-conjugation if
necessary. For a weight $\la$ in a certain subset $\Lambda_\chi$
of $\ev \h^*$ (see \eqref{weight}), we define the simple
$U_\chi(\h)$-module $V_\chi(\la)$ (which is in general not
one-dimensional because $\h$ is non-abelian), and then define the
baby Verma module
$
Z_{\chi}(\la)=\rea{\chi}{\g} \otimes_{\rea{\chi}{\h \oplus \n^+}}
V_{\chi}(\la).
$
Note that baby Verma modules have varied dimensions depending on
$\la$.

\subsection{}
Our first main result is the following criterion on the
irreducibility of $Z_\chi(\la)$ (see \eqref{eq:Phi} for the
precise definition of the polynomial $\Phi$).

\begin{thma} [Theorem~\ref{thm:semisimple}]
Assume that $\chi \in \ev \g^*$ is semisimple with $\chi(\ev
\n^+)=\chi(\ev \n^-) =0$. Then a baby Verma module $Z_{\chi}(\la)$
with $\la \in \La_{\chi}$ is irreducible if and only if $\Phi(\la)
\neq 0$.
\end{thma}
This should be regarded as a queer analogue of Rudakov's classical
result for modular Lie algebras \cite{Rud}. We need some extra care
to deal with the complication from the multi-dimensionality of the
high weight subspace $V_{\chi}(\la)$ of $Z_{\chi}(\la)$. An
immediate consequence of the above theorem is a neat criterion for
the semisimplicity of $\rea{\chi}{\g}$ associated to semisimple
$p$-characters $\chi$. Denote by $h_i \in \ev \h$ $(1\le i \le n)$
the  element corresponding  to the $i$th diagonal matrix unit of $A$
in \eqref{q(n)}.

\begin{thmb} [Theorem~\ref{semisimpleUg}]
Let $\chi$ be semisimple with $\chi(\ev \n^+)=\chi(\ev \n^-)=0$.
The algebra $\rea{\chi}{\g}$ is semisimple if and only if
$\,0\neq \chi(h_i)  \neq \pm \chi(h_j)$ for  $1 \leq i\neq j \leq n$.
\end{thmb}


\subsection{}

Another main result of the paper is the proof of the super KW conjecture,
which was first formulated in \cite{WZ}, for $\g =\qn$ with
nilpotent $p$-characters. Let $\chi\in \ev \g^*$ be nilpotent. We
regard $\chi \in \g^*$ by setting $\chi(\od \g) = 0$. Denote the
centralizer of $\chi$ in $\g$, which is clearly $\Z_2$-graded, by
$\g_{\chi} = \g_{\chi,\bar 0} + \g_{\chi,\bar 1}$. We show that $\dim \ev
\g - \dim \g_{\chi,\bar{0}} = \dim \od \g - \dim \g_{\chi,\bar{1}}$,
and this number
is actually an even integer say $2d$. The following theorem should be regarded as
a queer generalization of the celebrated Kac-Weisfeiler conjecture
(Premet's theorem \cite{Pr1}) for Lie algebras of reductive
algebraic groups.

\begin{thmc} [Theorem~\ref{th:superKW}]
 Let $\chi \in \ev \g^*$ be a nilpotent $p$-character. Then the dimension of
every simple $\rea{\chi}{\g}$-module is divisible by $\delta = p^d
2^d$.
\end{thmc}

The proof of the above theorem is similar to the one in \cite{WZ}
for basic classical Lie superalgebras, which in turn is a
generalization of the approach of \cite{Pr1} with some
modification using an idea from Skryabin \cite{Skr}. A
$\Z$-grading on $\g$ associated to $\chi$ is first constructed,
which leads to the construction of a $p$-nilpotent Lie subalgebra
$\m$ of $\g$. Then an elementary argument without using
support varieties shows that every simple
$\rea{\chi}{\g}$-module is free over $\rea{\chi}{\m}$, and the
above theorem follows now by noting that the dimension of
$\rea{\chi}{\m}$ is $\delta=p^d 2^d$.

The algebra $\rea{\chi}{\m}$ has a unique simple module $K_{\chi}$
which is one-dimensional. An extra bonus of the above proof is the
introduction of a $K$-superalgebra $W_\chi(\g)$ which will be called
the finite $W$-superalgebra of type $Q$ (see \cite{Pr2, WZ} for the
finite $W$-(super)algebras associated to the basic classical Lie
superalgebras including Lie algebras of reductive algebraic groups).
We show further that $\rea{\chi}{\g}$ is isomorphic to the matrix
algebra $M_{\delta} (W_\chi(\g)^{\rm op})$, and this provides a
conceptual explanation of the above $\delta$-divisibility theorem.
The complex counterpart of the algebra $W_\chi(\g)$ is expected to
have a rich representation theory and will be studied elsewhere \cite{Z}.

It is worth mentioning that Boe, Kujawa and Nakano have a similar 2-divisibility result
for $\g$-modules in characteristic zero \cite{BKN}.

\subsection{}

For $n=2$, we are able to analyze in detail the structures of the baby Verma modules and
the reduced enveloping algebras for $\mathfrak q (2)$. In various cases,
we work out  the structures of projective covers and the
blocks of $U_\chi(\mathfrak q(2))$-modules in terms of quivers. Remarkably,
the $\mathfrak q (2)$ case is far more involved than the classical
case of $\mathfrak{sl} (2)$ \cite{FP, Jan} or the $\mathfrak{osp}
(1|2)$ case treated in \cite{WZ}.

Let $\chi =\chi_s +\chi_n$ be a Jordan decomposition of a general
$p$-character $\chi$. In contrast to the cases of the simple Lie
algebras and basic classical Lie superalgebras, it is not possible
to regard the centralizer $\g_{\chi_s}$ of $\chi_s$ in $\g$ as a
Levi subalgebra and to fit it as a middle term of a triangular
decomposition of $\g$. Hence there is no natural generalization to
$\g$ of the functors which give rise to the Morita equivalence
\cite{FP} which reduces the study of modular representations of a
reductive Lie algebra with a general $p$-character to those of a
Levi subalgebra with a nilpotent $p$-character (also see \cite{WZ}
for a generalization to basic classical Lie superalgebras).
Nevertheless, we formulate a conjecture on the existence of a
puzzling Morita ``super-equivalence" for $\g$ with its centralizers
(without good candidates for adjoint functors), and prove it in the
case of $\mathfrak q(2)$ by ad hoc direct computations; see
Section~\ref{sec:MoritaSuper} for a precise definition of Morita
super-equivalence.

\begin{thmd} [Theorem~\ref{th:morita_q2}]
Let $\chi \in  \ev{\mathfrak q(2)}^*$ and $\chi =\chi_s +\chi_n$ be
its Jordan decomposition. Then the superalgebras $U_\chi(\mathfrak
q(2))$ and $U_\chi(\mathfrak q(2)_{\chi_s})$ are Morita
super-equivalent.
\end{thmd}

\subsection{}

The paper is organized as follows. In Section~\ref{sec:basics}, we
set up some basic notations and constructions on the queer Lie
superalgebra $\g$. In Section~\ref{sec:semisimple}, we establish our
results on the irreducibility of the baby Verma modules and the
semisimplicity of the algebra $U_\chi (\g)$. The super KW property
and related construction are presented in
Section~\ref{sec:nilpotent}. Sections~\ref{sec:q(2)I} and \ref{sec:q(2)II}
are devoted to a
detailed study of representations of $\mathfrak{q}(2)$. In Section~\ref{sec:conjecture}
we formulate a conjecture of Morita super-equivalence for $\qn$
with general $p$-characters and prove it
for $\mathfrak{q}(2)$.

{\em Convention}: By  subalgebras, ideals, modules, and submodules
etc. we mean in the super sense unless otherwise specified. The
graded dimension of a superspace $V =\ev V \oplus \od V$ will be
denoted by $\sudim V = \dim \ev V \vert \dim \od V$.


\section{The preliminaries}\label{sec:basics}

\subsection{The Lie superalgebra $\mathfrak{q}(n)$}

Let $K$ be an algebraically closed field of characteristic $p>2$.
Let $K^{m|n}$ denote the superspace of dimension $m|n$, and $\gl
(m|n)$ the Lie superalgebra of linear transformations of $K^{m|n}$.
Choosing a homogeneous basis of $K^{m|n}$ we may regard $\gl (m|n)$
as the superalgebra of $(m+n)\times(m+n)$ matrices. In the case when $m=n$
consider an odd automorphism $P:K^{n|n}\rightarrow K^{n|n}$ with
$P^2=-1$. The linear transformations of $\mathfrak{gl}(n|n)$
preserving $P$ is a subalgebra of $\mathfrak{gl}(n|n)$, denoted by
$\mathfrak{q}(n)$. We have
$\mathfrak{q}(n)=\mathfrak{q}(n)_{\bar{0}}\oplus
\mathfrak{q}(n)_{\bar{1}}$, with $\mathfrak{q}(n)_{\bar{0}}$
isomorphic to the general linear Lie algebra $\mathfrak{gl}(n)$ and
$\mathfrak{q}(n)_{\bar{1}}$ isomorphic to the adjoint module of
$\mathfrak{gl}(n)$. Choosing $P$ to be the $2n\times 2n$ matrix
\begin{equation*}
\begin{pmatrix}0&I_n\\ -I_n&0 \end{pmatrix}
\end{equation*}
with $I_n$ denoting the identity $n\times n$ matrix, we may
identify $\mathfrak{q}(n)$ as the subalgebra of
$\mathfrak{gl}(n|n)$ consisting of $2n\times 2n$ matrices of the
form \eqref{q(n)}. The even elements of
$\mathfrak{q}(n)$ are those for which $B=0$, while the odd
elements are those for which $A=0$.

From now on set $\g = \mathfrak{q}(n)$. The Lie superalgebra $\g$ admits an odd
nondegenerate $\g$-invariant symmetric bilinear form, which is given by
\[
(x,y):= \text{otr}(xy) \text{ for }x, y \in \g,
\]
where $xy$ denotes the matrix product, and $\text{otr}$ denotes the
odd trace given by
\[
\text{otr}\begin{pmatrix}
A&B\\
B&A\\
\end{pmatrix} =\text{trace}(B).
\]
It is known that all Cartan subalgebras of $\g$ are conjugate to
the Lie superalgebra $\h = \ev \h \oplus \od \h$ of
matrices~(\ref{q(n)}) with both $A$ and $B$ diagonal (which will be referred to
as the standard Cartan). All Borel (i.e. maximal solvable) Lie subalgebras
of $\g$ are conjugate to the standard Borel subalgebra consisting
of matrices~(\ref{q(n)}) with both $A$ and $B$ upper triangular.
The roots of $\g$ (i.e. elements $\alpha \in \ev \h^*$, for which
$\g_{\alpha}:= \{x\in \g|\; [h,x]=\alpha(h)x, \forall h\in \ev
\h\} \neq 0$) are the same as $\gl(n)$: if we let
$\{\varepsilon_i\}$ be a basis of $\ev \h^*$ dual to the standard
basis $\{h_i\}$ of $\ev \h$, where $h_i$ is of the form
(\ref{q(n)}) with the $i$th diagonal entry of $A$ being $1$ and
zero elsewhere, then the roots are
$$
\Delta = \{\varepsilon_i - \varepsilon_j |\; 1 \leq i \neq j \leq
n\}.
$$
The dimension of each root space is equal to $1|1$, in
contrast to the $\gl(n)$ case.

Let us fix some notations. In various places of this paper, we need
to work with some fixed Borel subalgebra $\mathfrak{b}$. It
determines a system of positive roots which will be denoted by
$\Delta^+$; the corresponding simple system is denoted by $\Pi$.
Also let $\mathfrak{n}^+ =\mathfrak{n}^+_{\bar 0}
+\mathfrak{n}^+_{\bar 1} $ (respectively $\mathfrak{n}^-$) denote
the Lie subalgebra of positive (respectively negative) root vectors.

\subsection{The reduced enveloping algebras}

Recall that (c.f. e.g. \cite{WZ}) a restricted Lie superalgebra is
a Lie superalgebra $\g =\ev \g +\od \g$ whose even subalgebra $\ev
\g$ is a restricted Lie algebra with $p$th power map ${}^{[p]}:
\ev \g \ra \ev \g$, and the odd part $\od \g$ is a restricted $\ev
\g$-module by the adjoint action. Let $\g$ be a restricted Lie
superalgebra and $V$ be a simple $\g$-module. The elements $x^p -
\pth x$ for $x \in \ev \g$ in the universal enveloping algebra
$U(\g)$ are central. Thus by Schur's lemma, they act as scalars
$\zeta(x)$ on $V$, which can be written as $\chi_V(x)^p$ for some
$\chi_V \in \ev \g^*$. We call $\chi_V$ the {\em $p$-character} of
the module $V$.

Fix $\chi \in \ev \g^*$. Let $I_{\chi}$ be the ideal of $U(\g)$
generated by the even central elements $x^p - \pth x - \chi(x)^p$
for all $x\in \ev \g$. The quotient algebra $\rea{\chi}{\g} :=
U(\g)/I_{\chi}$ is called {\em the reduced enveloping
superalgebra} with $p$-character $\chi$. A $\g$-module with
$p$-character $\chi$ is the same as a $\rea{\chi}{\g}$-module. We
often regard $\chi \in \g^*$ by letting
$\chi (\od \g) =0$. 

Recall $\ev \g =\gl(n)$. Any $p$-character $\tilde{\chi}$ is
$\text{GL}(n)$-conjugate to a $p$-character $\chi$ with $\chi(\ev
\n^+)=0$, and $\rea{\tilde{\chi}}{\g}\cong \rea{\chi}{\g}$. This
allows us to restrict ourselves for the rest of the paper to
consider only $p$-characters $\chi$ with $\chi(\ev \n^+)=0$. A
$p$-character $\chi \in \ev \g^*$ is called {\em semisimple} if it
is $\text{GL}(n)$-conjugate to some $\xi \in \ev \g^*$ with $\xi(\ev
\n^+) = \xi(\ev \n^-)=0$, and $\chi$ is called {\em nilpotent} if it
is $\text{GL}(n)$-conjugate to some $\eta \in \ev \g^*$ with
$\eta(\ev \n^+) = \eta(\ev \h)=0$. This could also be viewed
alternatively as follows: the odd bilinear form $(,)$ allows one to
identify $\ev \g^*$ with $\od \g$ which has the same underlying
space as $\gl(n)$. Then the $p$-character $\chi$ is semisimple
(respectively nilpotent) if and only if the corresponding element in
$\gl(n)$ is semisimple (respectively nilpotent).

\subsection{The baby Verma modules}
Fix a triangular decomposition
\[
\g =\mathfrak{n}^- \oplus \mathfrak{h} \oplus \mathfrak{n}^+,
\]
and let $\mathfrak{b}= \mathfrak{h} \oplus \mathfrak{n}^+$. For
$\la \in \h_{\bar{0}}^*$ we may consider the symmetric bilinear
form on $\h_{\bar{1}}$ defined by
$$
(a|b)_{\la}:=\la([a,b]), \quad a,b\in\h_{\bar{1}}.
$$
Now if $\h'_{\bar{1}}\subset\h_{\bar{1}}$ is a maximal isotropic
subspace with respect to this bilinear form and let $\od \h''$ be
a complement of $\od \h'$ in $\od \h$ (i.e. $\od \h = \od \h'
\oplus \od \h''$), we may extend $\la$ to a one-dimensional
representation $K_{\la}$ of
$\h_{\bar{0}}+\h'_{\bar{1}}$ by letting $\od \h'$ act trivially.

Let $\chi \in \ev \g^*$ be such that $\chi(\ev \n^+)=0$. Set
\begin{align}
\Lambda_{\chi} &= \{ \lambda \in \ev{\mathfrak{h}}^* \vert \;
\lambda(h)^p - \lambda(\pth h) = \chi(h)^p \text{ for all } h \in
\ev{\mathfrak{h}} \}   \nonumber
\\
&=\{ (\la_1, \ldots, \la_n) \vert\; \la_i^p - \la_i =\chi(h_i)^p,
1\ \leq i \leq n\},   \label{weight}
\end{align}
where $\la_i= \la(h_i)$. The module $K_{\lambda}$ is a
$\rea{\chi}{\ev \h \oplus \od \h'}$-module if and only if $\la \in
\La_{\chi}$. We define an irreducible $\rea{\chi}{\h}$-module
\[
V_{\chi}(\la) = \rea{\chi}{\h} \otimes_{\rea{\chi}{\ev \h \oplus \od
\h'}}K_{\la}, \qquad \la \in \La_{\chi}.
\]
This module has an odd automorphism (or say, is of type $Q$) if
and only if the dimension of the quotient space $\h_{\bar{1}}/{\rm
ker}(\cdot|\cdot)_{\la}$ is odd. We extend this irreducible
$\rea{\chi}{\h}$-module to an irreducible
$\rea{\chi}{\mathfrak{b}}$-module by letting $\n^+$ act trivially.
Inducing further we obtain the {\em baby Verma module} of
$\rea{\chi}{\g}$ associated to $\la\in\h^*_{\bar{0}}$
$$
Z_{\chi}(\la) =\rea{\chi}{\g} \otimes_{\rea{\chi}{\mathfrak{b}}}
V_{\chi}(\la).
$$
We denote   $v_0 =1 \otimes 1 \in Z_{\chi}(\la)$. As a vector space,
we have
$$Z_{\chi}(\la) \cong U_\chi ( \n^-)
\otimes V_\chi (\la).
$$


\subsection{$\rea{\chi}{\g}$ as symmetric algebra}

Recall that the supertrace of an endomorphism $X$ on a vector
space $\ev V \oplus \od V$ is defined to be $ \text{str}(X) =
\text{tr}(X|_{\ev V}) - \text{tr}(X|_{\od V}).$ An associative
superalgebra $A$ with a supersymmetric nondegenerate bilinear form
will be called a {\em symmetric (super)algebra}. One checks that
$
\text{str}(\ad x)=0,  \text{for all } x \in \ev \g.
$
Thus a variant of \cite[Prop.~2.7]{WZ} (also see \cite{FP})
applies and it gives the following.
\begin{proposition}\label{prop:sym}
The superalgebra $\rea{\chi}{\g}$ is symmetric for $\chi \in \ev
\g^*$.
\end{proposition}

\section{Modular representations with semisimple $p$-characters}\label{sec:semisimple}
Throughout this section, we assume that $\chi \in \ev \g^*$ is
semisimple with $\chi(\ev \n^+) = \chi(\ev \n^-)=0$. The goal of
this section is to establish  criterions for the irreducibility of
the baby Verma module $Z_\chi(\la)$ and for the semisimplicity of the
algebra $U_\chi(\g)$.

\subsection{Some $\mathfrak{q}(2)$ calculations}
\label{subsec:semi-q(2):an example}

Let us fix some notation for $\mathfrak{q}(2)$ first. We consider
the standard generators of $\mathfrak{q}(2)$: $e, f, h_1, h_2, E,
F, H_1, H_2$, described symbolically as
\[
\left(
\begin{array}{cc|cc}
h_1 & e & H_1 & E\\
f   & h_2&F   & H_2\\
\hline H_1 & E & h_1 & e\\
F & H_2 & f & h_2
\end{array}
\right).
\]
This description can be read in the following way: to each symbol
corresponds a matrix of $0$'s and $1$'s, in which the $1$'s are
situated precisely at the places occupied by the corresponding
symbol.
For $n,m \geq 0$, and $ k \geq 1$, denote
\begin{align*}
T_{n,k} = \tbinom{n}{k}\frac{(n-1)!}{(k-1)!}, \qquad
[x]_m   = x(x-1)\cdots (x-m+1).
\end{align*}

\begin{lemma}\label{lem:q(2)-2}
For $0 \leq a \leq p-1$, the following identity holds in
$\rea{\chi}{\mathfrak{q}(2)}$:
\begin{align}
e^a f &= a e^{a-1}(h_1 - h_2 + a-1) + f e^a,\notag\\
e f^a &= a f^{a-1} (h_1 - h_2 -(a-1)) + f^a e,\notag\\
e^a f^{a-1} &= T_{a,1} e[h_1-h_2+1]_{a-1}+ T_{a,2} fe^2[h_1-h_2+1]_{a-2}\notag\\
& \quad + \ldots + T_{a,a}f^{a-1}e^a,\tag{\ref{lem:q(2)-2}a}\\
e^a F &= a e^{a-1}(H_1-H_2) + (a-1)a e^{a-2}E +
Fe^a.\tag{\ref{lem:q(2)-2}b}
\end{align}
\end{lemma}
\begin{proof}
Follows by induction on $a$. The proof of (\ref{lem:q(2)-2}a) uses
the earlier formulas.
\end{proof}

\begin{lemma}\label{lem:top}
The following identity holds in the
$\rea{\chi}{\mathfrak{q}(2)}$-module $Z_{\chi}(\la)$:
\begin{equation*}
e^{p-1}E f^{p-1}F v_0 = (p-1)![h_1-h_2-1]_{p-1}(h_1 +h_2) v_0,
\end{equation*}
where $v_0=1 \otimes 1 \in Z_\chi(\la)$.
\end{lemma}

\begin{proof}
By a direct computation, we obtain the following identity in
$\rea{\chi}{\mathfrak{q}(2)}$:
\begin{equation*}
Ef^{p-1}F = f^{p-2}F(H_1 - H_2) + f^{p-1} (h_1 + h_2) - f^{p-1}FE.
\end{equation*}
Applying this to the hight weight vector $v_0$ gives us
\[
e^{p-1}E f^{p-1}F v_0 = e^{p-1} f^{p-2} F(H_1 - H_2) v_0 +
e^{p-1}f^{p-1}(h_1 +h_2) v_0.
\]
We shall compute the two summands on the right hand side. Using
(\ref{lem:q(2)-2}a) in the second identity below, we have
\[
\begin{split}
e^{p-1}f^{p-1} (h_1 +h_2) v_0 & = (h_1 +h_2) e^{p-1}f^{p-1} v_0\\
& = \left(h_1 +h_2) ( T_{p-1,1}e[h_1-h_2 +1]_{p-2}fv_0 + w e^2
fv_0
 \right) \\
&= T_{p-1,1}(h_1 +h_2) ef[h_1-h_2-1]_{p-2}v_0 +0\\
& =
T_{p-1,1} (h_1 +h_2) [(h_1-h_2)+fe][h_1-h_2-1]_{p-2}v_0\\
&=T_{p-1,1}(h_1 +h_2) [h_1-h_2]_{p-1}v_0,
\end{split}
\]
where $w$ is some vector in $\rea{\chi}{\mathfrak{q}(2)}$.
On the other hand, we have by (\ref{lem:q(2)-2}a-b) that
\[
\begin{split}
&e^{p-1} f^{p-2} F(H_1 - H_2) v_0 \\
&= T_{p-1,1}e[h_1-h_2
+1]_{p-2}F(H_1-H_2)v_0 + u e^2 F(H_1-H_2)v_0\\
&=T_{p-1,1}eF[h_1-h_2-1]_{p-2}(H_1-H_2)v_0\\
& \qquad
+ u [2E + 2e(H_1-H_2) +Fe^2](H_1-H_2)v_0\\
&= T_{p-1,1}[(H_1 -H_2)
+Fe][h_1-h_2-1]_{p-2}(H_1-H_2)v_0\\
&=T_{p-1,1}[h_1-h_2-1]_{p-2}(h_1+h_2)v_0,
\end{split}
\]
where $u$ is some vector in $\rea{\chi}{\mathfrak{q}(2)}$.

It follows by definition that $T_{p-1,1}=(p-1)!$ and $[x]_{p-1} +
[x-1]_{p-2} =[x-1]_{p-1}$. Now the lemma follows from combining
the above two computations.
\end{proof}

Define a polynomial $\phi$ in two variables $x,y$ as follows:
\begin{eqnarray} \label{phi}
\phi(x,y)=(x+y)(x-y-1)(x-y-2)\cdots(x-y-(p-1)).
\end{eqnarray}

\begin{proposition}\label{prop:q(2)-simplicity}
Let $\g=\mathfrak{q}(2)$. Assume that $\chi\in \ev \g^*$ is semisimple
satisfying $\chi(e) =\chi(f) =0$ and let $\la =(\la_1,\la_2)\in
\La_{\chi}$. Then the baby Verma module $Z_{\chi}(\lambda)$ is
simple if and only if $\phi(\lambda_1,\la_2) \neq 0$.
\end{proposition}
\begin{proof}
We use the special case of Lemma~\ref{lem:ss-lowest
vector} for $n=2$ (which can also be proved directly) asserting
that any nontrivial submodule of $Z_{\chi}(\la)$ contains the
vector $f^{p-1}F v_0$. Now by Lemma~\ref{lem:top},
$Z_{\chi}(\lambda)$ is simple  if and only if $e^{p-1}E f^{p-1}F
v_0$ is a nonzero multiple of $v_0$, if and only if
$\phi(\lambda_1,\la_2) \neq 0$.
\end{proof}

\subsection{An irreducibility criterion of $Z_\chi(\la)$}
We return to the general case for $\g=\mathfrak{q}(n)$. For the
rest of this section, $\h$ will be the standard Cartan subalgebra
with  basis $\{h_i, H_i\}_{1\le i\le n}$. Recall that $\Delta^+$
is the set of positive roots associated to a fixed triangular
decomposition $\g= \n^- \oplus \h \oplus \n^+$, and that the
definition of $Z_{\chi}(\la)$ depends on $\Delta^+$. For
$\la=(\la_1, \ldots, \la_n) \in \La_{\chi}$ with $\la_i=\la(h_i)$,
put
\begin{equation} \label{eq:Phi}
\Phi(\la) := \prod_{1 \leq i<j \leq n} \phi(\la_i, \la_j).
\end{equation}

\begin{theorem}\label{thm:semisimple}
Assume that $\chi \in \ev \g^*$ is semisimple with $\chi(\ev
\n^+)=\chi(\ev \n^-) =0$ and let $\la \in \La_{\chi}$. Then the
baby Verma module $Z_{\chi}(\la)$ is simple if and only if
$\Phi(\la) \neq 0$.
\end{theorem}
We need some preparations for the proof of the theorem. The height
of a root $\alpha \in \Delta^+$ is the sum of the coefficients in
the decomposition of $\alpha$ into simple roots. We index the
positive roots $\alpha_1, \ldots, \alpha_N,$ where
$
   N={n(n-1)}/{2},
$
enumerating first the roots of height $1$, then the roots of height
$2$, and so on.

For $\alpha = \varepsilon_k - \varepsilon_l$ ($k<l$), we use the
notation $e_{\alpha}$ (respectively $E_{\alpha}$) for the element of
the form~(\ref{q(n)}) with the $(k,l)$-entry of $A$ (respectively
$B$) being one and zero otherwise; also write $f_{\alpha}$
(respectively $F_{\alpha}$) for the element of the form~(\ref{q(n)})
with the $(l,k)$-entry of $A$ (respectively $B$) being one and zero
elsewhere. Further denote $e_i = e_{\alpha_i}$ (respectively $f_i =
f_{\alpha_i}$) and $E_i=E_{\alpha_i}$ (respectively $F_i =
F_{\alpha_i}$). Recall $N={n(n-1)}/{2}.$

\begin{lemma}\label{lem:ss-lowest vector}
Any nonzero submodule of a baby Verma module $Z_{\chi}(\la)$
contains the vector $f_1^{p-1}F_1f_2^{p-1}F_2\cdots f_N^{p-1}F_N
v_0$.
\end{lemma}
\begin{proof}
The proof is similar to that of \cite[Proposition~4]{Rud}. For the
sake of the reader, we outline the main steps. We show first that
\begin{align*}
f_j \cdot f_1^{i_1}F_1^{\ep_1} \cdots
f_{j-1}^{i_{j-1}}F_{j-1}^{\ep_{j-1}}f_j^{p-1}F_j^{\ep_{j}} f_{j+1}^{p-1}F_{j+1}
\cdots f_N^{p-1}F_N  v_0 &=0, \\
F_j \cdot f_1^{i_1}F_1^{\ep_1} \cdots
f_{j-1}^{i_{j-1}}F_{j-1}^{\ep_{j-1}}\ f_j^{i_j}\ F_j f_{j+1}^{p-1}F_{j+1}
\cdots f_N^{p-1}F_N  v_0 &=0,
\end{align*}
where $0\leq  i_s \leq p-1$ and $\ep_s =0,1$.
Then we show that
\begin{align*}
f_j \cdot   f_1^{i_1}F_1^{\ep_1}& \cdots
f_j^{i_j-1}F_j^{\ep_j}f_{j+1}^{p-1}F_{j+1}
\cdots f_N^{p-1}F_N  v_0  \\
&= f_1^{i_1}F_1^{\ep_1} \cdots f_j^{i_j}F_j^{\ep_j}
f_{j+1}^{p-1}F_{j+1}
\cdots f_N^{p-1}F_N  v_0, \\
F_j \cdot  f_1^{i_1}F_1^{\ep_1} & \cdots
f_j^{i_j}f_{j+1}^{p-1}F_{j+1}
\cdots f_N^{p-1}F_N  v_0  \\
 &= \pm f_1^{i_1}F_1^{\ep_1} \cdots
f_j^{i_j}F_j f_{j+1}^{p-1}F_{j+1}
\cdots f_N^{p-1}F_N  v_0,
\end{align*}
where $0\leq  i_s \leq p-1$ and $\ep_s =0,1$ for $s \leq j$. Now the
lemma follows easily from the above claims.
\end{proof}

A PBW  basis for a baby Verma module
$Z_{\chi}(\la)$ is given by
\[
f_1^{a_1}F^{\ep_1}\cdots f_N^{a_N}F_N^{\ep_N}Y_1^{\tau_1}\cdots
Y_{r}^{\tau_r}v_0 \quad  (0 \leq a_i \leq p-1; \; \ep_j, \tau_k=0,1)
\]
where $\{Y_1, \ldots Y_r\}$ is a basis for $\od \h''$ which we
recall is a complement of $\od \h'$ in $\od \h$. Let $\La(\od
\h'')_+$ be the linear span of $Y_1^{\tau_1}\cdots
Y_{r}^{\tau_r}$, not all $\tau_1, \ldots \tau_r$ being zero.

\begin{lemma}\label{lem:ss-F'}
Let $\la \in \La_{\chi}$. The following identity holds in
$Z_{\chi}(\la)$:
\[
e_1^{p-1}E_1\cdots e_N^{p-1}E_N\cdot f_1^{p-1}F_1\cdots
f_N^{p-1}F_Nv_0 =\tilde{\Phi}(\la_1, \ldots, \la_n)v_0 + w,
\]
for some polynomial $\tilde{\Phi}$ in $n$ variables of degree at
most $p^N$  and some   $w \in \La(\od \h'')_+  v_0$.
\end{lemma}
\begin{proof}
Follows by a weight consideration and Lemma~\ref{lem:top}.
\end{proof}

Recall the function $\phi$ is defined in \eqref{phi}.
\begin{lemma}\label{lem:ss-minimal parabolic}
Assume $\alpha =\varepsilon_i - \varepsilon_j$ is a simple root of
$\Delta^+$. If $\phi(\la_i, \la_j)=0$, then the baby Verma module
$Z_{\chi}(\la)$ is reducible.
\end{lemma}
\begin{proof}
For notational convenience, we assume without loss of generality
that $\Delta^+ =\{\vep_i - \vep_j \vert \; 1\leq i < j \leq n\}$
and that $\alpha= \vep_1 - \vep_2$. Then we may choose $\od \h'$,
$\od \h''$ and a basis $\{Y_1, \ldots, Y_r\}$ of $\od \h''$ to be
compatible with the embedded $\mathfrak{q}(2)$ corresponding to
the root $\vep_1 -\vep_2$, so that $\{Y_1\}$ is a basis for $\od
\h'' \cap \mathfrak{q}(2)$ and that $Y_1$ is orthogonal to the
span $ \widetilde{\h}_{\bar{1}}''$ of $Y_2, \ldots, Y_r$ with
respect to $(|)_{\la}$.

Now consider the minimal parabolic subalgebra $\p=\mathfrak q(2)
+\mathfrak{b}$, and the induced $\p$-module $Z^{\p}_{\chi}(\la)
=\rea{\chi}{\p}\otimes_{\rea{\chi}{\mathfrak{b}}} V_{\chi}(\la)$.
One can also write $\p = \mathfrak q(2) \oplus \widetilde{\h} \oplus
\widetilde{\n}^+$, where $\widetilde{\h}$ is the span of $h_i, H_i$
$(3\le i \le n)$, and $\widetilde{\n}^+$ is the span of all positive
root vectors except the ones for $\vep_1 - \vep_2$. Note that
$[\mathfrak q(2), \widetilde{\h}] =0$.
Since $\phi(\la_1,\la_2) = 0$, the baby Verma module $Z^{\mathfrak
q(2)}_{\chi}(\la_1,\la_2)$ of $\mathfrak q(2)$ is reducible by
Proposition~\ref{prop:q(2)-simplicity}. Then the $\p$-module
$Z^{\p}_{\chi}(\la)$ is also reducible, thanks to the identification
of the $\mathfrak p$-modules
$$Z^{\p}_{\chi}(\la) \cong
Z^{\mathfrak q(2)}_{\chi}(\la_1,\la_2) \otimes \La
(\widetilde{\h}_{\bar{1}}'')
$$
where the right-hand side carries a trivial action of $\widetilde{\n}^+$.
By the
transitivity of  induced modules we have
$$
Z_{\chi}(\la) \cong \rea{\chi}{\g}
\otimes_{\rea{\chi}{\p}}Z^{\p}_{\chi}(\la),
$$
and then the reducibility of $Z_{\chi}(\la)$ follows from the
reducibility of $Z^{\p}_{\chi}(\la)$.
\end{proof}

\begin{lemma}\label{lem:ss-divisor}
If $\phi(\la_i, \la_j)=0$ for some $1 \leq i \neq  j\leq n$, then
the baby Verma module $Z_{\chi}(\la)$ is reducible.
\end{lemma}
\begin{proof}
In this proof, we shall denote $\Delta^+$ and $\Pi$ by
$\Delta_{(0)}^+$ and $\Pi_{(0)}$ respectively, and write
$Z^{(0)}_{\chi}(\la)=Z_{\chi}(\la)$.
Let $\beta_1 \in \Pi^+_{(0)}$, and
$\Delta_{(1)}^+=s_{\beta_1}(\Delta_{(0)}^+)$. Let
$Z^{(1)}_{\chi}(\la)$ denote the baby Verma module with respect to
$\Delta^+_{(1)}$, that is, it is generated by a high weight vector
$v_0^{(1)}$ with respect to $\Delta^+_{(1)}$. Then
$e_{\beta_1}^{p-1}E_{\beta_1}v^{(1)}_0$ is a weight vector of
weight $\la$, and it is annihilated by  $e_{\alpha}, E_{\alpha}$ for
all $\alpha \in \Pi_{(0)}$. So there is a non-zero
$\g$-homomorphism
$
\psi_1: Z_{\chi}^{(0)}(\la) \ra Z_{\chi}^{(1)}(\la).
$

In general, we can find a sequence of positive roots $\beta_1,
\ldots, \beta_t$ such that $\beta_{k+1}$ $(0 \le k \le t-1$) is
a simple root for the positive system $\Delta_{(k)}^+ :=
s_{\beta_k}(\Delta_{(k-1)}^+)$, and that $\vep_i- \vep_j$ is a
simple root for the positive system $\Delta_{(t)}^+ =
s_{\beta_t}(\Delta_{(t-1)}^+)$. By the previous paragraph, there
exist non-zero $\g$-homomorphisms
$$
\psi_i: Z_{\chi}^{(i-1)}(\la) \ra Z_{\chi}^{(i)}(\la), \quad i=1,
\ldots, t.
$$
Since $Z_{\chi}^{(i-1)}(\la)$ and $Z_{\chi}^{(i)}(\la)$, $1 \leq i
\leq t$, have the same dimension, the reducibility of
$Z_{\chi}^{(i-1)}(\la)$ follows from the reducibility of
$Z_{\chi}^{(i)}(\la)$ via $\psi_i$. By Lemma~\ref{lem:ss-minimal
parabolic}, $Z_{\chi}^{(t)}$ is reducible, hence $Z_{\chi}(\la) =
Z_{\chi}^{(0)}(\la)$ is also reducible.
\end{proof}

\begin{proof}[Proof of Theorem~\ref{thm:semisimple}]
If $\Phi(\la)=0$, then $\phi(\la_i,\la_j)=0$ for some $1 \leq
i\neq j\leq n$. By Lemma~\ref{lem:ss-divisor}, $Z_{\chi}(\la)$ is
reducible. Moreover, by Lemmas~\ref{lem:ss-lowest vector}
and~\ref{lem:ss-F'}, $\tilde{\Phi}(\la)=0$. Hence, the polynomial $\tilde{\Phi}$
is always divisible by $\Phi$. Conversely, assume that
$Z_{\chi}(\la)$ is reducible. By Lemmas~\ref{lem:ss-lowest vector}
and~\ref{lem:ss-F'}, $\tilde{\Phi}(\la)=0$. Since $\Phi$ divides
$\tilde{\Phi}$ and $\deg \Phi \geq \deg{\tilde{\Phi}}$, we
conclude that $\Phi(\la)=0$.
\end{proof}
The following corollary is immediate from
Theorem~\ref{thm:semisimple}.
\begin{corollary}
Assume that $\chi \in \ev \g^*$ is semisimple with $\chi(h_1)=\cdots
= \chi(h_n)$ (for example, $\chi=0$). For $\la= (a, \ldots, a) \in \La_{\chi}$ with $a \neq
0$, the baby Verma module $Z_{\chi}(\la)$ is simple.
\end{corollary}

\subsection{A semisimplicity criterion of $U_\chi (\g)$}

\begin{theorem}\label{semisimpleUg}
Let $\chi$ be semisimple with $\chi(\ev \n^+)=\chi(\ev \n^-)=0$. The
algebra $\rea{\chi}{\g}$ is semisimple if and only if $\, 0 \neq\chi(h_i)
\neq \pm \chi(h_j)$ for all $1 \leq i\neq j \leq n$.
\end{theorem}

\begin{proof}
Since $\chi(\ev \n^+)=\chi(\ev \n^-)=0$, each baby Verma module has
a unique irreducible quotient, which will be denoted by
$L_\chi(\la)$. The simple $\rea{\chi}{\g}$-modules $L_{\chi}(\la)$
and $L_{\chi}(\la')$ for $\la \neq \la'$ are non-isomorphic, and so
there are $p^n$ simple $\rea{\chi}{\g}$-modules. By Wedderburn
Theorem and a dimension counting argument, $\rea{\chi}{\g}$ is
semisimple if and only if all the baby Verma modules $Z_{\chi}(\la)$
for  $\la \in \Lambda_{\chi}$ are simple (of type $Q$ for $n$ odd or
of type $M$ for $n$ even) and in addition all $\chi(h_i) \neq 0$.
Since $\la_k^p -\la_k = \chi(h_k)^p$ for each $k$, we have $(\la_i
\pm \la_j)^p - (\la_i \pm \la_j) = (\chi(h_i) \pm \chi(h_j))^p$.

If $\chi(h_i)  \neq \pm \chi(h_j)$ for all $i\neq j$,
then every $\la \in \La_{\chi}$ satisfies $\la_i \neq -\la_j$ and
$\la_i-\la_j \notin \mathbb{F}_p^*$ for all $i\neq j$.
So $\Phi(\la) \neq 0$, and by
Theorem~\ref{thm:semisimple}, $Z_{\chi}(\la)$ is simple for
$\la \in \Lambda_{\chi}$.

Conversely, assume $\chi(h_i)  = \pm \chi(h_j)$ for some $i \neq j$.
If $\chi(h_i) = \chi(h_j)$, then there exists $\la \in
\La_{\chi}$ such that $\la_i - \la_j \in \mathbb{F}_p^*$ (thanks
to the flexibility of choosing $\la$ by shifting the value of
$\la_i$ by any integer in $\mathbb{F}_p$). If
$\chi(h_i)= -\chi(h_j)$, then
there exists $\la \in \La_{\chi}$ such that $\la_i = -\la_j$. In
either case, we have $\Phi(\la)=0$. Thus by
Theorem~\ref{thm:semisimple}, $Z_{\chi}(\la)$ is reducible.
\end{proof}

\section{Modular representations with nilpotent $p$-characters}\label{sec:nilpotent}
\subsection{The centralizer of an odd nilpotent element}\label{subsec:Cen-Nil}

Let $\chi \in \ev \g^*$ be a $p$-character, and let $X\in \od \g$
be such that $\chi =(X,-)$.
Then the centralizer $\g_{\chi}$ of $\chi$ in $\g$ is identified
with the usual centralizer $\g_X$; that is, $\g_{X,\bar{0}}$
consists of matrices of the form (\ref{q(n)}) with $B=0$ and $A$
commuting with $X$ (viewed as a matrix), while $\g_{X,\bar{1}}$
consists of matrices of the form (\ref{q(n)}) with $A=0$ and $B$
anti-commuting with $X$.

Let $X \in \od \g$ be nilpotent. Up to a $GL(n)$-conjugation, we
can suppose that $X$ has the form (\ref{q(n)}) with $A=0$ and $B$
equal to the Jordan canonical form
\[
B=\begin{pmatrix} J_1 & &\\
& \ddots &\\
& & J_r
\end{pmatrix},
\]
where $J_i$ is a Jordan block of eigenvalue $0$ and size $d_i
\times d_i$, and $d_1 \ge d_2 \ge \ldots \ge d_r$.

\begin{proposition}\label{prop:comm-anticomm}
Let $X \in \od \g$ and $B$ be as above, and let $C=(C_{ij})$ be a matrix of
the same block type as $B$. Then,

\begin{enumerate}
\item
$C$ commutes with $B$ if and only if
\[
C_{ij}= \begin{pmatrix} a & b & \cdots & c\\
& a & \ddots & \vdots \\
& & \ddots & b\\
& & & a\\
0&0&0&0
\end{pmatrix} \text{  for }i\leq j,  \text{ or } C_{ij}=\begin{pmatrix}
0&a & b & \cdots & c\\
0&& a & \ddots & \vdots \\
0&& & \ddots & b\\
0&& & &  a
\end{pmatrix} \text{ for } i > j.
\]

\item
Also, $C$ anti-commutes with $B$ if and only if
\[
C_{ij}= \begin{pmatrix} a & b & \cdots & c\\
& -a & -b & \vdots \\
& & \ddots & \ddots\\
& & & \pm a\\
 0&0&0&0
\end{pmatrix} \text{  for }i\leq j,  \text{ or } C_{ij}=\begin{pmatrix}
0&a & b & \cdots & c\\
0&& -a & \ddots & \vdots \\
0&& & \ddots & \mp b\\
0&& & & \pm a
\end{pmatrix} \text{ for } i > j.
\]
\end{enumerate}
In particular, we have $\dim \g_{X,\bar{0}} = \dim \g_{X,\bar{1}}  =
\sum_{1\le i,j \le r} \text{min}\{d_i,d_j\}.$
\end{proposition}
\begin{proof}
%
The matrix $C$ commutes with $B$ if and only if
\[
J_i C_{ij} = C_{ij} J_j \quad \forall i,j.
\]
Also, $C$ anti-commutes with $B$ if and only if
\[
J_i C_{ij} = -C_{ij} J_j \quad \forall i,j.
\]
Then a direct computation shows that the $C_{ij}$ are of the  forms
as prescribed in the proposition. The dimension formula for $\dim
\g_{X, i}$ follows.
\end{proof}
\subsection{The $\Z$-grading}

Let $0\neq X \in \od \g$ be nilpotent. Recall $\ev \g =\od \g =\gl(n)$ under
the adjoint action of $GL(n)$. Then a standard construction of a $\Z$-grading
on $\gl(n) = \oplus_{k\in\Z} \gl(n) (k)$ (see \cite{Pr1} and \cite[Theorem~3.1]{WZ}) induces a $\Z$-grading
on $\g = \oplus_{k\in\Z} \g(k)$
which satisfies $\g(k)=\ev {\g(k)} \oplus
\od {\g(k)}$, $\ev {\g(k)} =
\od {\g(k)} =\gl(n)(k)$ and the following properties:
\begin{equation*}
X \in \g(2);
\end{equation*}
\begin{equation} \label{equ:grading2}
(\g(k), \g(l)) = 0, \quad \text{if }  k+l \neq 0;
\end{equation}
\begin{equation*}
\g_X = \oplus_{k \in \Z} \g_X(k) \qquad \text{where } \g_X(k)=\g_X
\cap \g(k);
\end{equation*}
\begin{equation*}
\g_X(s)=0 \qquad \forall \, s < 0;
\end{equation*}
\begin{equation} \label{equ:grading5}
\sudim \g_X = \sudim \g(0) + \sudim \g(1).
\end{equation}

\begin{example}
Let $n=4$. Let $X \in \od \g$  correspond to the Jordan block $J_4
\in \gl (4)$. The corresponding $H$ is the diagonal matrix
$\text{diag}\, (3,1,-1,-3)$.
Then the centralizer $\g_X$ consists of matrices of the form
\[
\left(
\begin{array}{cccc|cccc}
 x_0 & y_2 & z_4 & w_6 & a_0 & b_2 &c_4 &d_6\\
0_{-2} & x_0 & y_2 & z_4 & 0_{-2} & -a_0 & -b_2 & -c_4\\
 0_{-4} & 0_{-2} & x_0 & y_2 & 0_{-4} & 0_{-2} & a_0 & b_2\\
0_{-6} & 0_{-4} & 0_{-2} & x_0 & 0_{-6} & 0_{-4} & 0_{-2} & -a_0\\
\hline
a_0 & b_2 & c_4 & d_6 & x_0 & y_2 &z_4 &w_6\\
0_{-2} & -a_0 & -b_2 & -c_4 & 0_{-2} & x_0 & y_2 & z_4\\
 0_{-4} & 0_{-2} & a_0 & b_2 & 0_{-4} & 0_{-2} & x_0 &y_2\\
0_{-6} & 0_{-4} & 0_{-2} & -a_0 & 0_{-6} & 0_{-4} & 0_{-2} &
x_0\end{array} \right),
\]
where $x_i$ etc. are arbitrary scalars in $K$, $0_i = 0$, and the
index $i$ indicates the $\Z$-gradings of the corresponding matrix
entries. Clearly, $\dim \g_X =4|4$.
\end{example}

\subsection{Super KW property for nilpotent $p$-characters}
On $\g(-1)_{\bar 0}$ (respectively $\g(-1)_{\bar 1}$) there is a
non-degenerate symplectic (respectively symmetric) bilinear form
$\langle \cdot, \cdot\rangle$ given by
\begin{equation*}
\langle x, y \rangle := (X, [x, y]) = \chi([x, y]).
\end{equation*}
In other words, the above defines an {\em even} non-degenerate
skew-supersymmetric bilinear form $\langle \cdot, \cdot\rangle$ on
$\g(-1)$. Indeed, take a nonzero $x \in \g(-1)_i$ for $i \in \Z_2$.
Since $\g_X(s)=0$ unless $s \geq 0$, we have that $0 \neq [X, x] \in
\g(1)_{i+\bar{1}}$. By the non-degeneracy of the pairing between
$\g(1)_{i+\bar{1}}$ and $\g(-1)_i$, there exists $y \in \g(-1)_i$
with $0 \neq ([X,x], y) = (X, [x, y]) = \langle x, y \rangle$.

Take $\g(-1)'_i \subset \g(-1)_i$, where $i \in \Z_2$, to be a
maximal isotropic subspace with respect to $\langle\cdot,
\cdot\rangle$. Note that $\dim \g(-1)_i$ is even and $\dim
\g(-1)'_i = \dim \g(-1)_i/2$. Define a $p$-nilpotent Lie
subalgebra
\[
\m = \bigoplus_{k \geq 2} \g(-k) \bigoplus \g(-1)'.
\]
Then it follows by (\ref{equ:grading2}) and (\ref{equ:grading5})
that
\[\sudim
\m =\frac{1}{2} (\sudim \g - \sudim \g_{\chi}).
\]

\begin{proposition}\label{prop:nil-freeness}
Every $U_{\chi}(\g)$-module is $U_{\chi}(\m)$-free.
\end{proposition}
\begin{proof}
The proof is the same as the one for \cite[Proposition~4.2]{WZ},
which is in turn a superalgebra generalization of Skryabin
\cite[Theorem~1.3]{Skr}, thus is omitted.
\end{proof}

\begin{theorem}[Super KW property with nilpotent characters]
\label{th:superKW}
 Let $\chi \in \ev \g^*$ be nilpotent. Then,
$ \dim \ev \g - \dim \g_{\chi,\bar{0}} = \dim \od \g - \dim
\g_{\chi,\bar{1}}$ is an even number (denoted by $2d$), and the
dimension of every simple $\rea{\chi}{\g}$-module is divisible by
$\delta = p^d 2^d$.
\end{theorem}
\begin{proof}
The dimension equality follows from the equality  $\dim
\g_{\chi,\bar{0}} =\dim \g_{\chi,\bar{1}}$ in
Proposition~\ref{prop:comm-anticomm}. The divisibility of the
dimensions of simple $U_\chi(\g)$-modules is immediate from
Proposition~\ref{prop:nil-freeness}, by noting $\delta = \dim
\rea{\chi}{\m}$.
\end{proof}

Note that $U_\chi (\m)$ has a unique simple module, and this simple
module is one-dimensional and will be denoted by $K_\chi$. Denote by
$\mathcal{Q}_{\m}$ the induced $\rea{\chi}{\g}$-module
$\rea{\chi}{\g}\otimes_{\rea{\chi}{\m}}K_{\chi}$. We further define
the $K$-superalgebra
$$
W_\chi(\g)=\text{End}_{\rea{\chi}{\g}}(\mathcal{Q}_{\m}).
$$

\begin{theorem}
\begin{enumerate}
\item
The  $\rea{\chi}{\g}$-module $\mathcal{Q}_{\m}$ is
projective.

\item We have an isomorphism of superalgebras:
\[
\rea{\chi}{\g} \cong M_{\delta} (W_\chi(\g)^{\text op}).
\]
\end{enumerate}
Here $M_{\delta} (W_\chi(\g)^{\text op})$ denotes the matrix algebra
of size $\delta$ with entries in $W_\chi(\g)^{\text op}$.
\end{theorem}

\begin{proof}
The proof is the same as the one for \cite[Theorem~4.4]{WZ}, which
is a super generalization (with a mild modification of the proof which bypasses
completely the use of support variety)
of Premet \cite[Theorem~2.3 (i), (ii)]{Pr2}, thus is omitted.
\end{proof}

\begin{remark}
The algebra $W_{\chi}(\g)$, which is referred to as the finite $W$-superalgebra of
$\mathfrak{q}(n)$, admits a counterpart over the complex
field. It will be interesting to develop its structure
and representation theory.
\end{remark}

\section{The representation theory of $\mathfrak{q}(2)$, I}\label{sec:q(2)I}

In this and next sections, we study in detail the representation
theory of $\g = \mathfrak{q}(2)$. We still let $\h$, $\mathfrak{b}$  denote the
standard Cartan and Borel subalgebras of $\g$.
Let $\chi \in \ev {\mathfrak{q}(2)}^*$ be such that $\chi(e)=0$,
but for now we will not impose any condition on $\chi(f)$.
In this section, we shall determine the vectors in $Z_{\chi}(\la)$ annihilated by
$\n^+$ for every $\la \in \La_{\chi}$, which is
equivalent to describing all possible homomorphisms between baby
Verma modules.

\subsection{The case when $\la =(\la_1,\la_2)=0$}\label{subsubsec:q(2)-Verma-0}

In this case we have $\od \h'=\od \h$. So $Z_{\chi}(0)$ is induced
from the one-dimensional trivial $\rea{\chi}{\mathfrak{b}}$-module
$K_0$, and it has a basis
$
\{f^a F^{\ep} v_0 \vert \; 0 \leq a \leq p-1, \ep =0,1\},
$
where we denote by $v_0 = 1 \otimes 1$.

The action of $\g$ is given by:
\begin{align*}
h.f^aF^{\ep}v_0 & = -(a+\ep)(\vep_1-\vep_2)(h)f^aF^{\ep}v_0 &
\text{for } h \in
\ev \h,\\
H.f^a v_0 & = -a(\vep_1 - \vep_2)(H)f^{a-1}Fv_0,\\
H.f^a F v_0 & = (\vep_1+\vep_2)(H) f^{a+1} v_0 & \text{for } H \in
\od \h,\\
e.f^a F^{\ep} v_0 & = -a((a-1) + 2\ep) f^{a-1}F^{\ep} v_0,\\
E.f^a F^{\ep} v_0 & = (a-1)a (\ep-1) f^{a-2}F v_0.
\end{align*}

We collect a basis for the vectors annihilated by $\n^+$ as
follows: \vspace{.2cm}
\begin{center}
\begin{tabular}{|c|c|}
\hline Basis for vectors annihilated by $\n^+$ & Weights\\
\hline $v_0$ & \\
\cline{1-1} $f^{p-1}F v_0$  & $(0 ,0)$\\
\hline $f v_0$ & \\
\cline{1-1} $F v_0$ & $(-1, 1)$ \\
\hline
\end{tabular}
\end{center}

\subsection{The case when $\la_1 =\la_2 \neq
0$}\label{subsubsec:q(2)-Verma-equal}

Take $\od \h'= K(H_1 + \mu H_2),$ where $\mu\in K$ is such that
$\mu^2 = -1$. Then $V_{\chi}(\la)$ is two-dimensional with basis
$\{v_0= 1 \otimes 1_{\la}, v_1 = H_1 \otimes 1_{\la}\}$, and is of
type $M$. A basis of $Z_{\chi}(\la)$ is given by $\{f^a
F^{\ep}\otimes v_i \vert \; 0 \leq a \leq p-1; \ep ,i= 0,1\}$.

We record the action of $\g$ as follows.
{\allowdisplaybreaks
\begin{align*}
h \cdot f^a F^{\epsilon}  v_i & = (\lambda -(a +
\epsilon)\alpha)(h) f^a F^{\epsilon}  v_i  \\
H_1 \cdot f^a F   v_0 & = -f^aF  v_1 + f^{a+1} v_0 \\
H_1 \cdot f^a v_0 & = f^a   v_1 - a f^{a-1}F v_0 \\
H_1 \cdot f^a F v_1 & = -\lambda_1 f^a F v_0 +
f^{a+1} v_1 \\
H_1 \cdot f^a v_1 & = \lambda_1 f^a v_1 - a
f^{a-1}F v_1 \\
H_2 \cdot f^a F v_0 & = {\mu}^{-1} f^aF v_1 +
f^{a+1} v_0 \\
H_2 \cdot f^a v_0 & = -{\mu}^{-1}f^a v_1 + a
f^{a-1}F v_0 \\
H_2 \cdot f^a F v_1 & = \mu \lambda_2 f^a F v_0 +
f^{a+1}  v_1 \\
H_1 \cdot f^a v_1 & = -\mu \lambda_2 f^a v_1 - a f^{a-1}F  v_1 \\
 \\
e \cdot f^a F v_0 & = [-a(a+1)+a(\lambda_1-\lambda_2)]f^{a-1}F  v_0
+ (1+
{\mu}^{-1}) f^a  v_1 \\
e \cdot f^a v_1 & = [a(\lambda_1 - \lambda_2)-(a-1)a]
f^a  v_1 \\
e \cdot f^a F v_1 & = [-a(a+1)+a(\lambda_1-\lambda_2)]f^{a-1}F  v_1
+ (\lambda_1+
\mu\lambda_2) f^a v_0 \\
e \cdot f^a v_0 & = [a(\lambda_1 - \lambda_2)-(a-1)a] f^{a-1}  v_0  \\
 \\
E \cdot f^a F v_0 & = -a(1+{\mu}^{-1})f^{a-1}F
v_1 + (\lambda_1 + \lambda_2) f^a v_0 \\
E \cdot f^a v_1 & = -(a-1)a f^{a-2}F v_1 + a(\lambda_1 + \mu \lambda_2)f^{a-1} v_0 \\
E \cdot f^a F v_1 & = -a(\lambda_1+\mu \lambda_2)f^{a-1}F
 v_0 + (\lambda_1+\lambda_2) f^a v_1 \\
E \cdot f^a v_1 & = -(a-1)a f^{a-2}F v_0 + a(1+{\mu}^{-1})f^{a-1}
 v_1.
\end{align*}
 }

A basis for the vectors annihilated by $\n^+$ is $\{v_0, v_1\}$.

\subsection{The case when $\la_1 =-\la_2 \neq
0$}\label{subsubsec:q(2)-Verma-opposite}

Take $\od \h' =K(H_1+H_2)$. Then $V_{\chi}(\la)$ is
two-dimensional with basis $\{v_0= 1 \otimes 1_{\la}, v_1 = H_1
\otimes 1_{\la}\}$, and is of type $M$. A basis of $Z_{\chi}(\la)$
is given by $\{f^a F^{\ep}\otimes v_i \vert \; 0 \leq a \leq p-1;
\ep ,i= 0,1\}$. The action of $\g$ is given by the same formula as
in Sect.~\ref{subsubsec:q(2)-Verma-equal}, with $\mu=1$.

A basis for the vectors annihilated by $\n^+$ is given as follows,
where the vectors with weight $(\la_2, \la_1)$ can happen if and
only if $\la_1 \in \mathbb{F}_p^*$. We use $\star$ in the table here
and similar situations below to indicate a conditional existence.
\vspace{.2cm}
\begin{center}
\begin{tabular}{|c|c|}
\hline Basis for vectors annihilated by $\n^+$ & Weights\\
\hline $v_0$ & \\
\cline{1-1} $v_1$ & $(\la_1,\la_2)$\\
\hline $F  v_1$ &\\
\cline{1-1} $f v_1 - \la_1 F  v_0$ & $(\la_1-1,
-\la_1+1)$\\
\hline $f^{2\la_1-1} F v_1$ & \\
\cline{1-1} $4\la_1f  v_1 - f^{2\la_1}F v_0$ & $ (\la_2,
\la_1) \quad \star$  \\
\hline
\end{tabular}
\end{center}

\subsection{The case when $0 \neq \la_1^2 \neq \la_2^2 \neq 0$}
\label{subsubsec:q(2)-Verma-square-noneq}

Let $\mu = -\la_1^2/ \la_2^2$ and $\od \h' =K(H_1 + \mu H_2)$.
Then $V_{\chi}(\la)$ is two-dimensional with basis $\{v_0= 1
\otimes 1_{\la}, v_1 = H_1 \otimes 1_{\la}\}$, and is of type $M$.
A basis of $Z_{\chi}(\la)$ is given by $\{f^a F^{\ep}\otimes v_i
\vert \; 0 \leq a \leq p-1; \ep ,i= 0,1\}$. The action of $\g$ is
given by the same formula as in
Sect.~\ref{subsubsec:q(2)-Verma-equal}.

Let $b$ be the integer satisfying $0\leq b <p$ and $b \equiv \la_1
-\la_2-1 \pmod{p}$. A basis for the vectors annihilated by $\n^+$ is
given as follows, where the vectors with weight $(\la_2, \la_1)$ can
happen if and only if $\la_1 -\la_2 \in \mathbb{F}^*_p$.

\begin{center}
\begin{tabular}{|c|c|}
\hline Basis for vectors annihilated by $\n^+$ & Weights\\
\hline $v_0$ & \\
\cline{1-1} $v_1$  & $(\la_1 ,\la_2)$\\
\hline $(b+1)f^{b} F v_0 - (1+ {\mu}^{-1})f^{b+1} v_1 $ & \\
\cline{1-1} $(b+1)f^b F v_1 - (\la_1+ \mu\la_2)f^{b+1}
v_0$ & $(\la_2, \la_1) \quad \star$ \\
\hline
\end{tabular}
\end{center}

\subsection{The case when $\la_1 =0, \la_2 \neq 0$}
\label{subsubsec:q(2)-Verma-zero-nonzero}

Take $\od \h' =KH_1$. The irreducible
$\rea{\chi}{\mathfrak{b}}$-module $V_{\chi}(\la)$ is
two-dimensional of type $Q$ with basis $\{v_0= 1 \otimes 1_{\la},
v_1 = H_2 \otimes 1_{\la}\}$. A basis of $Z_{\chi}(\la)$ is given
by $\{f^a F^{\ep}\otimes v_i \vert \; 0 \leq a \leq p-1; \ep ,i=
0,1\}$.

The action of $\g$ is given by:
{\allowdisplaybreaks
\begin{align*}
h \cdot f^a F^{\epsilon} v_i & = (\lambda -(a +
\epsilon)\alpha)(h) f^a F^{\epsilon} v_i  \\
H_1 \cdot f^a F v_0 & = f^{a+1} v_0 \\
H_1 \cdot f^a  v_0 & =- a f^{a-1}F  v_0 \\
H_1 \cdot f^a F  v_1 & = f^{a+1} v_1 \\
H_1 \cdot f^a v_1 & =-a f^{a-1}F  v_1 \\
H_2 \cdot f^a F   v_0 & =- f^aF  v_1 + f^{a+1} v_0 \\
H_2 \cdot f^a  v_0 & = f^a v_1 + a f^{a-1}F v_0 \\
H_2 \cdot f^a F  v_1 & = - \lambda_2 f^a F  v_0 +
f^{a+1}  v_1 \\
H_2 \cdot f^a  v_1 & =  \lambda_2 f^a  v_0 + a f^{a-1}F  v_1 \\
 \\
%
e \cdot f^a F  v_0 & = [-a(a+1)+-a\lambda_2]f^{a-1}F
v_0 - f^a  v_1 \\
e \cdot f^a   v_1 & = [-a \lambda_2-(a-1)a]
f^a  v_1 \\
e \cdot f^a F  v_1 & = [-a(a+1)-a\lambda_2]f^{a-1}F
v_1 - \lambda_2 f^a v_0 \\
e \cdot f^a  v_0 & = [ -a \lambda_2-(a-1)a] f^{a-1} v_0
\end{align*}

\begin{align*}
E \cdot f^a F  v_0 & = f^{a-1}F
v_1 + \lambda_2 f^a v_0 \\
E \cdot f^a v_1 & = -(a-1)a f^{a-2}F v_1 - a \lambda_2f^{a-1} v_0 \\
E \cdot f^a F v_1 & = a\lambda_2f^{a-1}F
 v_0 + \lambda_2 f^a v_1 \\
E \cdot f^a  v_1 & = -(a-1)a f^{a-2}F v_0 - a f^{a-1} v_1.
\end{align*}
}

Let $b$ be the integer satisfying $0\leq b <p$ and $b \equiv
-\la_2-1 \pmod{p}$. A basis for the vectors annihilated by $\n^+$ is
as follows, where the vectors with weight $(\la_2, 0)$ can appear in
the annihilator if and only if $\la_2 \in \mathbb{F}^*_p$.
\begin{center}
\begin{tabular}{|c|c|}
\hline Basis for vectors annihilated by $\n^+$ & Weights\\
\hline $v_0$ & \\
\cline{1-1} $v_1$  & $(0 ,\la_2)$\\
\hline $-\la_2 f^b F v_0 +f^{b+1}
 v_1 $ & \\
\cline{1-1} $2(b+1)f^b F v_1 + \la_2 f^{b+1} v_0$ & $(\la_2, 0)  \quad \star$ \\
\hline
\end{tabular}
\end{center}

\subsection{The case when $\la_1 \neq 0, \la_2= 0$} This is
similar to \ref{subsubsec:q(2)-Verma-zero-nonzero} and is thus omitted.

\section{The representation theory of $\mathfrak{q}(2)$, II}\label{sec:q(2)II}

In this section, we will study the structures of $U_\chi(\mathfrak{q}(2))$
and its blocks.

Recall that for an associative superalgebra $A$, a simple
$A$-supermodule $N$ is of type $Q$ (respectively of type $M$) if
$\text{End}_A(N)$ is $2$-dimensional (respectively $1$-dimensional),
or equivalently if $N$ admits (respectively does not admit) an odd
automorphism.

\subsection{Structure of $\rea{\chi}{\g}$ for semisimple $\chi$}

Assume that $\chi$ is semisimple with $\chi(e)=\chi(f)=0$. We now
use the information from Section~\ref{sec:q(2)I} and in addition
that $\chi(f)=0$ to analyze in detail the structure of
$Z_{\chi}(\la)$ and then of $\rea{\chi}{\g}$.

\subsubsection{$0 \neq \chi(h_1)^2 \neq \chi(h_2)^2 \neq
0$}\label{q(2)-rea-ss-noneq-nonzero}

It follows from the results of
\ref{subsubsec:q(2)-Verma-square-noneq} that these baby Verma
modules are irreducible of type $M$, and are pairwise
non-isomorphic. By dimension consideration, we conclude that the
algebra $\rea{\chi}{\g}$ is semisimple. Of course, this is
consistent with Theorem~\ref{semisimpleUg}.

\subsubsection{$\chi(h_1)=\chi(h_2) \neq 0$}\label{subsubsec:q(2)-rea-ss-eq-nonzero} The
high weights $\la \in \La_{\chi}$ are divided into two cases:
\begin{itemize}
\item[(i)]
$\la_1 = \la_2$. There are $p$ such weights. The baby Verma modules are
as in \ref{subsubsec:q(2)-Verma-equal}, and they are irreducible of
type $M$.

\item[(ii)] $\la_1-\la_2 \in \mathbb{F}_p^*$. There are $p(p-1)$
such weights. The baby Verma module $Z_{\chi}(\la_1,\la_2)$ (see
\ref{subsubsec:q(2)-Verma-square-noneq}) is reducible with a
unique submodule $L_{\chi}(\la_2,\la_1)$ of high weight $(\la_2,
\la_1)$ and dimension $d$, where $d$ is determined by $1 \leq d <4p$
and $d \equiv 4(\la_2-\la_1)\pmod{p}$. Both the submodule
$L_{\chi}(\la_2, \la_1)$ and the quotient $L_{\chi}(\la_1,\la_2)$ of
$Z_{\chi}(\la_1,\la_2)$ are irreducible of type $M$.
\end{itemize}

The results of Holmes and Nakano \cite{HN} apply in
our setup, since all the simple modules $L_{\chi}(\lambda)$ are of
type $M$. In particular, by \cite[Thms.~4.5 and 5.1]{HN} the
projective cover $P_{\chi}(\lambda)$ of $L_{\chi}(\lambda)$ has a
baby Verma filtration, and for any $\lambda , \mu \in \La_{\chi}$
one has the Brauer type reciprocity
$
( P_{\chi}(\lambda): Z_{\chi}(\mu) ) = [Z_{\chi}(\mu) :
L_{\chi}(\lambda)],
$
where $( P_{\chi}(\lambda): Z_{\chi}(\mu) )$ is the multiplicity of
$Z_{\chi}(\mu)$ appearing in the baby Verma filtration of
$P_{\chi}(\lambda)$, and $[Z_{\chi}(\mu) : L_{\chi}(\lambda)]$ is
the multiplicity of $L_{\chi}(\lambda)$ in a composition series of
$Z_{\chi}(\mu)$.
Hence,

(i) $P_{\chi}(\la_1,\la_1)$ are simple;

(ii) if $\la_1-\la_2 \in \mathbb{F}_p^*$, then $(P_{\chi}(\la_1,\la_2):
Z_{\chi}(\mu)) =1$ for $\mu = (\lambda_1,\la_2)$ and
$(\la_2,\la_1)$, and $0$ otherwise.

\begin{lemma}\label{lem:q(2)-ss-equal}
For $\la=(\la_1,\la_2) \in \La_{\chi}$ with $\la_1-\la_2 \in
\mathbb{F}_p^*$, the radical series of $P_{\chi}(\la)$ is as
follows.
\begin{enumerate}
\item  $\text{head } P_{\chi}(\la_1,\la_2)= \text{rad}^2 P_{\chi}(\la_1,\la_2) =\text{soc } P_{\chi}(\la_1,\la_2)=
L_{\chi}(\la_1,\la_2)$.
\item  $\text{rad } P_{\chi}(\la_1,\la_2)/\text{rad}^2
P_{\chi}(\la_1,\la_2)= L_{\chi}(\la_2,\la_1)\oplus
L_{\chi}(\la_2,\la_1).$
\end{enumerate}
\end{lemma}
\begin{proof}
Since $\rea{\chi}{\g}$ is a symmetric algebra, we can
argue similarly as \cite[proof of Proposition~5.1.3]{Ger}. The
details will be omitted here.
\end{proof}


\begin{proposition}\label{prop-q(2)-ss-equal}
Let $\g =\mathfrak{q}(2)$, and let $\chi \in \ev\g^*$ be semisimple
such that $\chi(e)=\chi(f)=0$ and $\chi(h_1) = \chi(h_2) \neq 0$. Then
\begin{itemize}
\item[(i)] For each $(\la_1,\la_1) \in \La_{\chi}$, the baby Verma module $Z(\la_1,\la_1)$ is
projective and simple.

\item[(ii)] For $\la=(\la_1,\la_2) \in
\La_{\chi}$ with $\la_1 \neq \la_2$, there is a block with exactly two simple modules
$L_{\chi}(\la_1,\la_2)$ and $L_{\chi}(\la_2,\la_1)$,
and this block is isomorphic to the algebra
given by the quiver
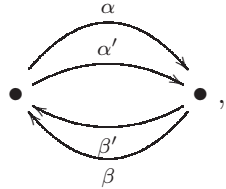
\[\xymatrix@C=2cm{
\bullet
 \ar@/^1pc/[r]^{\alpha'}
 \ar@/^2pc/@<.7ex>[r]^{\alpha}
& \bullet
 \ar@/^1pc/[l]^{\beta'}
 \ar@/^2pc/[l]^{\beta}
},\] with relations $\alpha\circ \beta =\beta \circ
\alpha=\alpha'\circ \beta' =\beta' \circ \alpha'=0$, $\alpha'\circ
\beta = \alpha\circ \beta'$, and $\beta'\circ \alpha=\beta \circ
\alpha'$.
\end{itemize}
\end{proposition}

\begin{proof}
Only the last assertion of (ii) on Morita equivalence needs an
explanation. The quiver and most of relations can be read off from
Lemma~\ref{lem:q(2)-ss-equal}. To get all of the relations, one
constructs some projective modules explicitly in a similar fashion
as in Xiao \cite[Section~2.2]{Xia}, then one shows they are indeed
projective covers $P_{\chi}(\la)$. From there, one obtains all
relations of the quiver since the homomorphisms between projective
covers can be explicitly read off.
\end{proof}

\subsubsection{$\chi(h_1)=-\chi(h_2) \neq 0$}
\label{subsubsec:q(2)-rea-ss-opp-nonzero}

The high weights $\la \in \La_{\chi}$ are divided into two
cases:
\begin{itemize}
\item[(i)] $\la_1=-\la_2 \notin \mathbb{F}_p$. There are $p$ such weights.
The baby Verma module $Z(\la_1, \la_2)$ (see
\ref{subsubsec:q(2)-Verma-opposite}) is reducible with a
unique submodule $L_{\chi}(\la_1-1,\la_2+1)$ of dimension $2p$. Both the submodule
$L_{\chi}(\la_1-1,\la_2+1)$ and the quotient $L_{\chi}(\la_1,\la_2)$
are irreducible of type $M.$

\item[(ii)] $\la_1 \neq -\la_2$. There are $p(p-1)$ such weights. The
baby Verma modules (see
\ref{subsubsec:q(2)-Verma-square-noneq}) are irreducible of
type $M$.
\end{itemize}

Again, the Brauer type reciprocity holds in this case. Hence,

(i) if $\la_1=-\la_2 \notin \mathbb{F}_p$, then
$(P_{\chi}(\la_1,\la_2): Z_{\chi}(\mu)) =1$ for $\mu =
(\lambda_1,\la_2) \text{ or } (\la_1+1,\la_2-1)$, and is $0$
otherwise;

(ii) if $\la_1 \neq -\la_2$, then
$P_{\chi}(\la_1,\la_2)=Z_{\chi}(\la_1,\la_2)=L_{\chi}(\la_1,\la_2)$.

The next lemma follows from this and that
$\rea{\chi}{\g}$ is a super-symmetric algebra.

\begin{lemma}\label{lem:q(2)-ss-opposite}
We have
\begin{itemize}
\item[(1)] $\text{head }P_{\chi}(\la_1,-\la_1)\cong \text{rad}^2 P_{\chi}(\la_1,-\la_1)=\text{soc
}P_{\chi}(\la_1,-\la_1)=L_{\chi}(\la_1,-\la_1)$.
\item[(2)] $\text{rad }P_{\chi}(\la_1,-\la_1)/\text{rad}^2
 P_{\chi}(\la_1,-\la_1) \cong L_{\chi}(\la_1-1,-\la_1+1) \oplus
L_{\chi}(\la_1+1,-\la_1-1)$.
\end{itemize}
\end{lemma}

\begin{proposition}\label{prop-q(2)-ss-opposite}
Let $\g=\mathfrak{q}(2)$, and let $\chi\in \ev \g^*$ be semisimple
with $\chi(h_1)=-\chi(h_2) \neq 0$. Then
\begin{itemize}
%
\item[(i)] the $p$ simple modules $L_{\chi}(\la_1,-\la_1)$ with $(\la_1,-\la_1) \in \La_{\chi}$
are $2p$-dimensional and belong to the same
block. This block is isomorphic to the quiver algebra
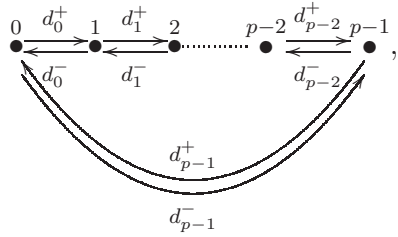
\[\xymatrix{
*!<0ex,-.8ex>{\stackrel{0}{\bullet}}
 \ar@<.4ex>[r]^{d_0^+}
 \ar@<-.4ex>@{<-}[r]_{d_0^-}
 \ar@/_4pc/@<-1.5ex>[rrrr]_{d_{p-1}^-}
& *!<0ex,-.8ex>{\stackrel{1}{\bullet}}
 \ar@<.4ex>[r]^{d_1^+}
 \ar@<-.4ex>@{<-}[r]_{d_1^-}
& *!<0ex,-.8ex>{\stackrel{2}{\bullet}}
 \ar@{..}[r]
& *!<0ex,-.8ex>{\stackrel{p-2}{\bullet}}
 \ar@<.4ex>[r]^{d_{p-2}^+}
 \ar@<-.4ex>@{<-}[r]_{d_{p-2}^-}
& *!<0ex,-.8ex>{\stackrel{p-1}{\bullet}}
 \ar@/^4pc/@<+.5ex>[llll]_{d_{p-1}^+}
},\] with relations $(d^+)^2=(d^-)^2=d^+d^-+d^-d^+=0$, where
$d^{\pm}=\sum_{l \in \mathbb{F}_p} d_l^{\pm}$;

\item[(ii)] each $Z_{\chi}(\la_1,\la_2)$ with $\la_1 \neq -\la_2$
is projective and simple of type $M$.
\end{itemize}
\end{proposition}
\begin{proof}
The case (ii) is clear by using Brauer reciprocity.
Part (i) follows from Lemma~\ref{lem:q(2)-ss-opposite} and a
similar argument as the proof of \cite[Theorem~5.2.1]{Ger}.
\end{proof}

\subsubsection{$\chi(h_1)=0, \chi(h_2)\neq 0$}
\label{subsubsec:q(2)-rea-ss-zero-nonzero}

The high weights $\la \in \La_{\chi}$ are divided into two
cases:
\begin{itemize}
\item[(i)] $\la_1=0$. There are $p$ such weights. The baby Verma
modules (see \ref{subsubsec:q(2)-Verma-zero-nonzero}) are
irreducible of type $Q$.
\item[(ii)] $\la_1 \neq 0$. There are $p(p-1)$ such weights. The baby
Verma modules (see \ref{subsubsec:q(2)-Verma-square-noneq}) are irreducible of type $M$.
\end{itemize}

Note in both cases we have
$Z_{\chi}(\la_1,\la_2)=L_{\chi}(\la_1,\la_2)$.

The structure theorem of associative superalgebras can be used to
estimate the dimensions of projective covers $P_{\chi}(\la_1,\la_2)$
of irreducible modules $L_{\chi}(\la_1,\la_2)$. To be precise, the
dimension of $P_{\chi}(\la_1,\la_2)$ equals the number of
composition factors of $\rea{\chi}{\g}$ isomorphic to
$L_{\chi}(\la_1,\la_2)$ if $L_{\chi}(\la_1,\la_2)$ is of type $M$,
and equals twice the number if it is type $Q$.
By the exactness of the functor
$\rea{\chi}{\g}\otimes_{\rea{\chi}{\mathfrak{b}}}-$, the number of
composition factors of $\rea{\chi}{\g}$ isomorphic to $Z_{\chi}
(\la_1,\la_2)$ equals the number of composition factors of
$\rea{\chi}{\mathfrak{b}}$ isomorphic to $V_{\chi} (\la_1,\la_2)$.
This number is $4p$ for all $\la \in \La_{\chi}$.

The dimension of $P_{\chi}(\la_1,\la_2)$ is $8p$ in case (i), and is
$4p$ in case (ii). In case (i), $P_{\chi}(\la_1,\la_2)$ are not
simple and they have a simple head
$Z_{\chi}(\la_1,\la_2)=L_{\chi}(\la_1,\la_2)$. On the other hand,
$\rea{\chi}{\g}$ is a (super-)symmetric algebra. Thus
$P_{\chi}(\la_1,\la_2)$ will have
$Z_{\chi}(\la_1,\la_2)=L_{\chi}(\la_1,\la_2)$ as its socle. We
conclude that $P_{\chi}(\la_1,\la_2)$ is a self-extension of
$L_{\chi}(\la_1,\la_2)$. As a result, the endomorphism ring
$\text{End}_{\rea{\chi}{\g}}(P_{\chi}(\la_1,\la_2))$ is isomorphic
to the ring $K[x]/\langle x^2\rangle$, where $x$ corresponds to the
projection of $P_{\chi}(\la_1,\la_2)$ to its socle. In case (ii), we
have $P_{\chi}(\la_1,\la_2)=L_{\chi}(\la_1,\la_2)$, since they have
the same dimension. Put
$$T=\bigoplus_{\la_1=0} P_{\chi}(\la_1,\la_2) \bigoplus
\bigoplus_{\la_1 \neq 0} P_{\chi}(\la_1,\la_2)^2,
$$
where for a module
$M$, $M^r$ denotes the direct sum of $r$ copies of $M$. The left
regular module $\rea{\chi}{\g}$ is isomorphic to $T^{2p}$ and
\begin{align*}
\rea{\chi}{\g} &\cong
\text{End}_{\rea{\chi}{\g}}(\rea{\chi}{\g})^{\text op} \cong
\text{End}_{\rea{\chi}{\g}}(T^{2p})^{\text op} \cong
(M_{2p}(\text{End}_{\rea{\chi}{\g}} (T)))^{\text op}   \notag\\
 & \cong (\oplus_{\la_1=0}M_{2p}(\mathfrak{q}_{1}(K[x]/\langle x^2\rangle)\bigoplus
\oplus_{\la_1\neq 0}M_{4p}(K))^{\text op}   \notag\\
 & \cong (\oplus_{\la_1=0}\mathfrak{q}_{2p}(K[x]/\langle x^2\rangle)\bigoplus
\oplus_{\la_1\neq 0}M_{4p}(K))^{\text op},
\end{align*}
where $\mathfrak{q}_n(K)$ denotes the simple associative superalgebra
consisting of all $2n\times 2n$ matrices of the form (\ref{q(n)}).
In summary, we have proved the following.
\begin{proposition}\label{prop:q(2)-ss-zero-nonzero}
Let $\g=\mathfrak{q}(2)$. Let $\chi\in \ev \g^*$ be semisimple with
$\chi(h_1)=0$ and $\chi(h_2) \neq 0$. Then,
\begin{itemize}
%
\item[(i)] every baby Verma module is irreducible:
$Z_{\chi}(\la_1,\la_2)$ is of type $M$ for $\la_1 \neq
0$, and $Z_{\chi}(\la_1, \la_2)$ is of type $Q$ for $\la_1 =0$.

\item[(ii)]
as algebras, $\rea{\chi}{\g} \cong (\oplus_{\la_1=0}
\mathfrak{q}_{2p}(K[x]/\langle x^2\rangle)\bigoplus
\oplus_{\la_1\neq 0}M_{4p}(K))^{\text op}$.
\end{itemize}
\end{proposition}

\subsection{Structure of $Z_{\chi}(\la)$ with a mixed $p$-character}
\label{subsubsec:q(2)-rea-mixed}

Let $\chi(h_1) = \chi(h_2) \neq 0$, and $\chi(f)=1$. The high
weights $\la \in \La_{\chi}$ are divided into two cases:
\begin{itemize}
\item[(i)] $\la_1 =\la_2$. There are $p$ such weights. The baby Verma
modules (see \ref{subsubsec:q(2)-Verma-equal}) are
irreducible of type $M$ and pairwise non-isomorphic.
\item[(ii)]
$\la_1 \neq \la_2$. There are $p(p-1)$ such weights. The baby Verma
modules (see \ref{subsubsec:q(2)-Verma-square-noneq}) are
irreducible of type $M$. We have $Z_{\chi}(\la_1,\la_2) \cong
Z_{\chi}(\la_2,\la_1)$ and there is no other isomorphism among
these baby Verma modules.
\end{itemize}

Arguing similarly as in \ref{subsubsec:q(2)-rea-ss-zero-nonzero}, we
prove the following.

\begin{proposition}\label{prop:q(2)-mixed}
Let $\g=\mathfrak{q}(2)$. Let $\chi \in \ev \g^*$ be such that
$\chi(h_1) = \chi(h_2) \neq 0$ and $\chi(f)=1$. Then,
\begin{itemize}
\item[(i)] every baby Verma module is simple, $4p$-dimensional and  of type $M$.

\item[(ii)] for $(\la_1,\la_1) \in \La_{\chi}$,
the baby Verma module  $Z_{\chi}(\la_1,\la_1)$ is
projective.

\item[(iii)] for
$(\la_1,\la_2) \in \La_{\chi}$ with $\la_1 \neq \la_2$, the projective
cover is a self-extension of
$Z_{\chi}(\la_1,\la_2)$.

\item[(iv)] as algebras, $\rea{\chi}{\g} \cong M_{4p}(K)^{\oplus p} \oplus M_{4p}(K[x]/\langle
x^2 \rangle)^{\oplus \frac{p(p-1)}{2}}.$
\end{itemize}
\end{proposition}

\subsection{Structures of $\rea{0}{\g}$-modules}

Let $\chi=0$. We shall drop the index $\chi$ or $0$ for the baby
Verma, projective and simple modules of $\rea{0}{\g}$.


We artificially divide the baby Verma modules into the following.

\begin{itemize}
\item[(i)] $(\la_1,\la_2)=(0,0)$. The baby Verma module (see
\ref{subsubsec:q(2)-Verma-0}) has a unique submodule, $L
(p-1,1-p)$, of dimension $(2p-2)$, while the irreducible
quotient $L (0,0)$ is two-dimensional.

\item[(ii)] $(\la_1, -\la_1), \la_1 \neq 0$.
There are $(p-1)$ such weights.  By analyzing the vectors
annihilated by $\ev \n^+$, we see that each baby Verma module $Z (\la_1,-\la_1)$ (see
\ref{subsubsec:q(2)-Verma-opposite}) has a
composition series of four simple modules
\[
L(\la_1, -\la_1),\; L (p-1-\la_1, 1-p+\la_1),\; L (\la_1-1,-\la_1
+1),\;  L (-\la_1, \la_1).
\]
The dimension of $L (\la_1,-\la_1)$ is the number $b$ determined
by the conditions $0 \leq b<2p $ and $b \equiv (4 \la_1 -2)
\pmod{2p}$.

\item[(iii)] $(\la_1, \la_1), \la_1 \neq 0$.
There are $(p-1)$ such weights. The baby Verma modules (see
\ref{subsubsec:q(2)-Verma-equal}) are simple of type $M$.

\item[(iv)] $(0,\la_2), \la_2 \neq 0$.
There are $(p-1)$ such weights.
By examining the vectors
annihilated by  $\ev \n^+$, we see that the baby Verma module
$Z (0,\la_2)$ (see \ref{subsubsec:q(2)-Verma-zero-nonzero}) has a
simple head $L (0,\la_2)$ and a simple socle $L
(\la_2,0)$, both $2p$-dimensional and of type $Q$.

\item[(v)] $(\la_1, 0), \la_1 \neq 0$. This case is similar to
case (iv), thus omitted.

\item[(vi)] $(\la_1,\la_2)$ with $0\neq \la_1^2 \neq \la_2^2 \neq
0$. There are $(p-1)(p-3)$ such weights. Each baby Verma module $Z (\la_1,\la_2)$
(see  \ref{subsubsec:q(2)-Verma-square-noneq}) has
a unique submodule $L (\la_2,\la_1)$, which is simple of dimension
$d$, where $d$ is determined by $0 \leq d <4p$ and $ d \equiv
4(\la_2-\la_1) \pmod{p}$. The head $L (\la_1,\la_2)$ is
simple of dimension $4p-d$.
\end{itemize}


By the exactness of the functor
$\rea{\chi}{\g}\otimes_{\rea{\chi}{\mathfrak{b}}}-$, the number of
composition factors of $\rea{\chi}{\g}$ isomorphic to $Z
(\la_1,\la_2)$ equals the number of composition factors of
$\rea{\chi}{\mathfrak{b}}$ isomorphic to $V (\la_1,\la_2)$. This
number is $8p$ for weight $(0,0)$, and $4p$ otherwise. The
structures of baby Verma modules have been given explicitly above.
From this we conclude that
\begin{equation*}
\dim P (\la_1,\la_2)= \begin{cases} 16p, &\text{if }
\la_1=\la_2 =0;\\
16p, &\text{if } \la_1 =-\la_2 \neq 0;\\
4p, &\text{if } \la_1=\la_2 \neq 0;\\
16p, &\text{if } \la_1=0, \la_2 \neq 0;\\
16p, &\text{if } \la_1\neq 0, \la_2=0;\\
8p, &\text{if } 0 \neq \la_1^2 \neq \la_2^2 \neq 0.
\end{cases}
\end{equation*}

From this we further conclude that $Z (\la_1,\la_2)$ is projective
if $\phi(\la_1,\la_2) \neq 0$ and $\la_i \neq 0$, $i=1,2$. In
particular, $Z (a,a)$ is projective and simple for $a \in \mathbb
F_p^*$ as claimed in Theorem~\ref{thm:semisimple}.

\subsection{Structure of $\rea{\chi}{\g}$-modules with $\chi$
nilpotent}

Assume $\chi(e)=\chi(h_1)=\chi(h_2)=0$, and $\chi(f)=1$. Since
$\chi$ is of standard Levi form (cf. \cite[Definition~10.1]{Jan}),
each baby Verma module $Z_\chi(\la)$ has a unique irreducible
quotient $L_\chi(\la)$ (cf. \cite[Proposition~10.2]{Jan}, the same
argument applies here). As in the (restricted) case when $\chi =0$,
we divide the baby Verma modules according to their high weights as
follows.
\begin{itemize}
\item[(i)] $(\la_1,\la_2)=(0,0)$. The baby Verma module (see
\ref{subsubsec:q(2)-Verma-0}) is simple of type $M$. We
have an isomorphism
$
L_{\chi}(0,0) \cong L_{\chi}(p-1,1-p).
$

\item[(ii)] $(\la_1, -\la_1), \la_1 \neq 0$. There are $(p-1)$ such
weights. By analyzing the vectors
annihilated by $\ev \n^+$ (see
\ref{subsubsec:q(2)-Verma-opposite}), the baby Verma modules
$Z_{\chi}(\la_1,-\la_1)$ has
a simple socle $L_{\chi}(-\la_1,\la_1)$ and a simple head
$L_{\chi}(\la_1,-\la_1)$, each of dimension $2p$. We have an isomorphism
$
L_{\chi}(\la_1,-\la_1) \cong L_{\chi}(p-1-\la_1,1-p+\la_1).
$

\item[(iii)] $(\la_1, \la_1), \la_1 \neq 0$. There are $(p-1)$ such weights.
The baby Verma modules (see
\ref{subsubsec:q(2)-Verma-equal}) are simple of type $M$.

\item[(iv)] $(0,\la_2), \la_2 \neq 0$. There are $(p-1)$ such weights.
The baby Verma modules (see
\ref{subsubsec:q(2)-Verma-zero-nonzero}) are
simple of type $Q$. The vectors annihilated by $\ev \n^+$ given
in \ref{subsubsec:q(2)-Verma-zero-nonzero} provide us an
isomorphism
$
Z_{\chi}(0,\la_2) \cong Z_{\chi}(\la_2,0).
$

\item[(v)] $(\la_1, 0), \la_1 \neq 0$. This case is similar to
case (iv), and the baby Verma modules $Z_{\chi}(\la_1,0)$ are
simple of type $Q$.

\item[(vi)] $(\la_1,\la_2)$ with $0\neq \la_1^2 \neq \la_2^2 \neq
0$. There are $(p-1)(p-3)$ such weights. The baby Verma modules
(see \ref{subsubsec:q(2)-Verma-square-noneq}) are simple
of type $M$. We have an isomorphism
$
Z_{\chi}(\la_1,\la_2) \cong Z_{\chi}(\la_2,\la_1).
$
\end{itemize}

By the same argument as in the previous subsection, we
estimate the dimensions of projective covers
$P_{\chi}(\la_1,\la_2)$ of $L_{\chi}(\la_1,\la_2)$ as follows:
\begin{equation*}
\dim P_{\chi}(\la_1,\la_2)= \begin{cases} 16p & \text{for }
L_{\chi}(\la_1,-\la_1)\cong L_{\chi}(p-1-\la_1,1-p+\la_1),\la_1 \in \mathbb{F}_p;\\
8p & \text{for } L_{\chi}(\frac{p-1}{2},\frac{p+1}{2})\cong L_{\chi}(\frac{p+1}{2},\frac{p-1}{2});\\
4p & \text{for }L_{\chi}(\la_1,\la_1), \la_1\neq 0;\\
16p & \text{for } L_{\chi}(0,\la_2) \cong L_{\chi}(\la_2,0);\\
8p & \text{for } L_{\chi}(\la_1,\la_2) \cong L_{\chi}(\la_2, \la_1),
0 \neq \la_1^2 \neq \la_2^2 \neq 0.
\end{cases}
\end{equation*}

For this, we conclude that $Z_{\chi}(\la_1,\la_1)$ are projective
simple for $\la_1 \neq 0$. Since $\rea{\chi}{\g}$ is a symmetric
algebra by Proposition~\ref{prop:sym}, $P_{\chi}(\la_1,\la_2)$ is
a self-extension of $L_{\chi}(\la_1,\la_2)=Z_{\chi}(\la_1,\la_2)$
for $0 \neq \la_1^2 \neq \la_2^2 \neq 0$.

\section{Modular representations with general $p$-characters}\label{sec:conjecture}

\subsection{The centralizer of an odd element} For a general
odd element $X \in \od \g$, let $X= X_s + X_n$ be its Jordan
decomposition (which is understood via the identification $\od \g
=\gl(n)$). As in the Lie algebra setup, we clearly have
\[
\g_{X,\bar{0}}= \g_{X_s, \bar{0}} \cap \g_{X_n, \bar{0}}.
\]
A much less trivial relation holds for the odd parts of
the corresponding centralizers.

\begin{lemma}  \label{lem:odd-cent}
Let $X=X_s+X_n$ be the Jordan decomposition of an odd element $X \in
\od \g$. Then we have
\begin{equation} \label{even-cent}
\g_{X,\bar{1}}= \g_{X_s, \bar{1}} \cap \g_{X_n, \bar{1}}
\end{equation}
and thus $\g_{X}= \g_{X_s} \cap \g_{X_n}.$
\end{lemma}
\begin{proof}
Without lost of generality, we assume that $X$ is of Jordan
canonical form
\[
X=\begin{pmatrix} J_{d_1}(\la_1) & &\\
& \ddots &\\
& & J_{d_r}(\la_r)
\end{pmatrix},
\]
where $J_{d_i}(\la_i)$ denotes the $d_i \times d_i$-Jordan block
with $\la_i$ on the diagonal. Then by Horn-Johnson
\cite[Theorem~4.4.11]{HJ}, the dimension of $\g_{X,\bar{1}}$ is
given by
\[
\dim \g_{X,\bar{1}} = \sum_{\la_i=-\la_j} \min\{d_i, d_j\}.
\]
Recall the description of $\g_{X_n,\bar{1}}$ in
Sect.~\ref{subsec:Cen-Nil}. On the other hand, an element in
$\g_{X_s, \bar{1}}$ has the form (\ref{q(n)}) with $A=0$ and
\[
B=\begin{pmatrix} B_{11} &  \cdots & B_{1r}\\
\vdots & \ddots & \vdots\\
B_{r1}& \cdots  & B_{rr}
\end{pmatrix},
\]
where the $d_i \times d_j$-matrix $B_{ij}=0$ if $\la_i \neq
-\la_j$ and $B_{ij}$ is arbitrary if $\la_i= -\la_j$. It follows
that the dimension of $\g_{X_s, \bar{1}} \cap \g_{X_n, \bar{1}}$
is $\sum_{\la_i=-\la_j} \min\{d_i, d_j\}$, which is same as $\dim
\g_{X, \bar{1}}$ given above. Obviously $\g_{X_s, \bar{1}} \cap
\g_{X_n, \bar{1}} \subseteq \g_{X,\bar{1}}$, so $\g_{X,\bar{1}}=
\g_{X_s, \bar{1}} \cap \g_{X_n, \bar{1}}.$
\end{proof}

Assume now that the odd element $X \in \od \g$ is semisimple and
hence is $\text{GL}(n)$-conjugate to some
element $Y\in \od \g$ of the form (\ref{q(n)}) with $A=0$ and
\[
B=\text{diag}(\underbrace{0,\ldots,0}_m,\underbrace{\mu_1,\ldots,
\mu_1}_{r_1},\underbrace{-\mu_1,\ldots,-\mu_1}_{s_1},\ldots,
\underbrace{\mu_t,\ldots,\mu_t}_{r_t},\underbrace{-\mu_t,\ldots,-\mu_t}_{s_t}),
\]
where $\mu_1, \ldots, \mu_t$ are squarely distinct nonzero scalars,
and $m, r_i, s_i \geq 0$. The next lemma follows by a direct computation.

\begin{lemma} \label{lem:oddcent}
For a semisimple odd element $X \in \od \g$ as above, the centralizer $\g_X$
is isomorphic to a direct sum $\mathfrak{q}(m)\oplus \gl(r_1|s_1)\oplus \cdots
\oplus \gl(r_t|s_t)$.
\end{lemma}

\subsection{A conjecture of Morita super-equivalence}
\label{sec:MoritaSuper}

Given a finite dimensional
superalgebra $A$, we denote by $A$-$\mathfrak{mod}$ the category of
finite-dimensional $A$-modules and $\Irr(A)$ the set of
isoclasses of simple $A$-supermodules.

\begin{conjecture}\label{conj:morita}
Let $\chi \in \ev \g^*$ be a $p$-character with Jordan decomposition
$\chi = \chi_s + \chi_n$. 
Let $b_i = \dim \g_i - \dim \g_{\chi_s,i}$ for $i \in \Z_2$. Then
there are adjoint exact functors $F$ and $G$:
\[
\rea{\chi}{\g}\text{-}\mathfrak{mod}\;\begin{picture}(22,10)

\put(0,3){$\longrightarrow$}

\put(0,-2){$\longleftarrow$}

\put(5,-7){$\scriptstyle{G}$}

\put(6,9){$\scriptstyle{F}$}

\end{picture}\;
 \rea{\chi}{\g_{\chi_s}}\text{-}\mathfrak{mod}
\]
satisfying the following.
\begin{itemize}
\item[(i)] Suppose $b_1$ is even. Then $F$ and $G$ are inverse
equivalences of categories, inducing a type-preserving bijection
between  $\Irr (\rea{\chi}{\g})$ and $\Irr
(\rea{\chi}{\g_{\chi_s}})$. Moreover, for a
$\rea{\chi}{\g_{\chi_s}}$-module $V$, $\dim G(V) =p^{\frac{b_0}{2}}
2^{\frac{b_1}{2}} \dim V.$

\item[(ii)] Suppose $b_1$ is odd. Then
\[
F \circ G \cong \text{Id}\oplus \Pi \quad G \circ F \cong
\text{Id}\oplus \Pi,
\]
where $\Pi$ is the parity change functor in a module category of a
superalgebra. The functor $F$ induces a bijection of $\Irr
(\rea{\chi}{\g})$ of type $M$ (respectively, of type $Q$) and $\Irr
(\rea{\chi}{\g_{\chi_s}})$ of type $Q$ (respectively, of type $M$).
Moreover, for $V \in \Irr (\rea{\chi}{\g_{\chi_s}})$ of type $M$,
the dimension of the corresponding $\rea{\chi}{\g}$-module $G(V)$ is
$p^{\frac{b_0}{2}} 2^{\frac{b_1+1}{2}} \dim V;$ while for $V \in
\Irr (\rea{\chi}{\g_{\chi_s}})$ of type $Q$,  the dimension of
$G(V)$ is $p^{\frac{b_0}{2}} 2^{ \frac{b_1-1}{2} } \dim V.$
\end{itemize}
\end{conjecture}
In the above and later on, $\chi$ in $\rea{\chi}{\g_{\chi_s}}$ is
understood as the restriction of $\chi$ to ${\g_{\chi_s}}$. One can
show by Lemma~\ref{lem:oddcent} that
$$
b_1 \equiv  \# \{1\le i \le n \mid \chi (h_i) \neq 0\} \mod 2.
$$
We will say the superalgebras $\rea{\chi}{\g}$ and
$\rea{\chi}{\g_{\chi_s}}$ are {\em Morita super-equivalent} if they
satisfy the properties prescribed in the above conjecture. In case~
(i) above, the superalgebras are indeed Morita equivalent the usual
sense.

\begin{remark}
When $\g$ is one of the basic classical Lie superalgebras, the above
Morita super-equivalence is indeed the usual Morita equivalence with
explicitly given functors (\cite[Theorem~5.2]{WZ}). This in turn was
a generalization of a theorem of Friedlander and Parshall \cite{FP}
(also cf. \cite{KW}) for Lie algebras of reductive algebraic groups.
However, $\g =\qn$ does not admit a natural triangular decomposition
with $\g_{\chi_s}$ as the middle term. This is already evident from
the calculation of $\mathfrak q(2)$ below. Hence, the natural
adjoint functors in \cite{FP} (also \cite{WZ}) have no counterpart
in the current setup.

Note that there is a similar result of Frisk and Mazorchuk \cite{FM}
in characteristic zero which establishes a super-equivalence between
the strongly typical blocks of the category $\mathcal O$ of $\qn$
and those of its even subalgebra $\gl(n)$. It is interesting to see
if it is possible to adapt their method to the modular setting.
\end{remark}

\subsection{The Morita super-equivalence for $\mathfrak{q}(2)$}

\begin{theorem}  \label{th:morita_q2}
Conjecture~\ref{conj:morita} holds for $\mathfrak q(2)$. That is,
the algebras $U_\chi(\mathfrak q(2))$ and $U_\chi(\mathfrak q(2)_{\chi_s})$
are Morita super-equivalent.
\end{theorem}

This subsection is devoted to the proof of the above theorem
by a detailed analysis of the representation theory of the centralizers of
semisimple part $\chi_s$ of $p$-characters $\chi$ in the case of
$\mathfrak{q}(2)$ and then a comparison with the results in Section~\ref{sec:q(2)II}.

\subsubsection{Semisimple $\chi$ with $0 \neq \chi(h_1)^2 \neq \chi(h_2)^2 \neq 0$}

The centralizer $\g_{\chi_s}$ is the even Cartan subalgebra $\ev
\h$. The algebra $\rea{\chi}{\ev \h}$ is
semisimple and commuative. Thus $\rea{\chi}{\ev \h}$ and
$\rea{\chi}{\g}$ are Morita equivalent, by a comparison with
 \ref{q(2)-rea-ss-noneq-nonzero}.

\subsubsection{Semisimple $\chi$ with $\chi(h_1)=\chi(h_2) \neq 0$}
\label{subsubsec:q(2)-cen-eq-nonzero}

The centralizer $\g_{\chi_s}$ is the even subalgebra $\ev \g
=\gl(2)$. Its reduced enveloping algebra $\rea{\chi}{\ev
\g}=\rea{\chi}{\gl(2)}$ is isomorphic to $\rea{0}{\mathfrak{sl}(2)}
\otimes K[x]/{(x^p-x-\chi(h_1)^p)}$.
By combining with the well-known structure of the algebra $\rea{0}{\mathfrak{sl}(2)}$
(see for example \cite[Propostion~2.4]{FP} and
\cite[Example~3.10]{Gor}), we deduce the following.
\begin{proposition}
Let $\chi \in \ev{\gl(2)}^*$ be semisimple with $\chi(h_1) =
\chi(h_2) \neq 0$. Then,
\begin{itemize}
%
\item[(i)] the baby Verma $\rea{\chi}{\gl(2)}$-module with   $\la_1-\la_2=p-1$ is
projective and simple;

\item[(ii)] for $\la_1-\la_2\in \mathbb F_p \backslash \{p-1\}$,
there is one block with exactly two simple $\rea{\chi}{\gl(2)}$-modules
of high weights $(\la_1,\la_2)$ and $(\la_2-1,\la_1+1)$.
This block is isomorphic to the
quiver algebra given in
Proposition~\ref{prop-q(2)-ss-equal}~(ii).
\end{itemize}
\end{proposition}

By a comparison with Proposition~\ref{prop-q(2)-ss-equal},
$\rea{\chi}{\g}$ and $\rea{\chi}{\g_{\chi_s}}$ are
Morita equivalent.

\subsubsection{Semisimple $\chi$ with $\chi(h_1)=-\chi(h_2) \neq 0$}

The centralizer $\g_{\chi_s}$ is spanned by $h_1, h_2, E$, and $F$
and it is isomorphic to the Lie superalgebra $\gl(1|1)$.
Denote by  $\tilde{\mathfrak{b}}$ the subalgebra $Kh_1 \oplus Kh_2
\oplus KE$.
We have all irreducible $\rea{\chi}{\tilde{\mathfrak{b}}}$-module
given by $K_\la = K$ with $\la(h_i)^p-\la(h_i)=\chi(h_i)^p$, upon
which $h_i$ acts as a scalar $\la_i$, and $E$ acts as zero.
Inducing from $K_\la$, we get the baby Verma modules for
$\rea{\chi}{\gl(1|1)}$:
\[
\tilde{Z}_{\chi}(\la) = \rea{\chi}{\gl(1|1)}
\otimes_{\rea{\chi}{\tilde{\mathfrak{b}}}} K_{\la}
\]
which is two-dimensional and has a unique simple quotient $\tilde{L}_{\chi}(\la)$.
%
\begin{proposition}
Let $\chi \in \ev {\gl(1|1)}^*$ be such that $\chi(h_1)=-\chi(h_2)
\neq 0$. Then,
\begin{itemize}
%
\item[(i)] the $p$ simple modules $\tilde{L}(\la_1, -\la_1)$
belong to a single
block,  and this block is isomorphic to the quiver algebra with
relations given in Proposition~\ref{prop-q(2)-ss-opposite} (ii);

\item[(ii)] each baby Verma   $\tilde{Z}_\chi(\la_1,\la_2)$ with $\la_1+\la_2\neq 0$
is projective and simple.
\end{itemize}
\end{proposition}

\begin{proof}
When $\la_2=-\la_1$, the baby Verma
module $\tilde{Z}_\chi(\la_1,-\la_1)$ is reducible and has a unique ($1$-dimensional)
submodule of  weight $(\la_1 -1, -\la_1+1)$. A projective cover of $\tilde{L}(\la_1, -\la_1)$
can be constructed
explicitly (similar to and simpler than the $\mathfrak{sl}(2)$ case \cite{Xia}). This
leads to the calculation of the underlying block in terms of quivers.

The remaining case with $\la_1+\la_2\neq 0$ is easy.
\end{proof}

By a comparison with Proposition~\ref{prop-q(2)-ss-opposite},
$U_\chi (\g)$ and $U_\chi (\g_{\chi_s})$
are Morita equivalent.

\subsubsection{Semisimple $\chi$ with $\chi(h_1)=0$ and $\chi(h_2)\neq 0$}


The centralizer $\g_{\chi_s}$ is $\ev \h \oplus K H_1 \cong
\mathfrak q(1) \oplus K h_2$. The weights $\la \in \ev \h^*$ such
that $\la(h_i)^p-\la(h_i)=\chi(h_i)^p$ can be divided into two
cases:
\begin{itemize}
\item[(i)] $\la_1 = 0$. Then the relations
$
h_1 v =0,  h_2 v = \la_2 v$ and $ H_1 v =0
$
define a one-dimensional $\rea{\chi}{\g_{\chi_s}}$-module.

\item[(ii)] $\la_1 \neq 0$. The irreducible $\rea{\chi}{\ev
\h}$-modules are one-dimensional of the form $Kv$ upon which $h_i$
act as scalars $\la_i$. We have the induced
$\rea{\chi}{\g_{\chi_s}}$-module:
\[
\tilde{Z}_{\chi}(\la_1, \la_2) = \rea{\chi}{\g_{\chi_s}}
\otimes_{\rea{\chi}{\ev \h}} Kv
\]
which are irreducible of type $Q$.
\end{itemize}

By a parallel analysis as in
\ref{subsubsec:q(2)-rea-ss-zero-nonzero}, we have an algebra isomorphism
\[
\rea{\chi}{\g_{\chi_s}} \cong (\oplus_{\la_1=0}M_1(K[x]/\langle
x^2\rangle)\bigoplus \oplus_{\la_1\neq 0}\mathfrak{q}_{1}(K))^{\text
op}.
\]
Recall that the algebra $U_\chi (\g)$ was computed in
Proposition~\ref{prop:q(2)-ss-zero-nonzero} and it is indeed Morita equivalent to $U_\chi (\g_{\chi_s})$.

\subsubsection{A mixed case: $\chi(h_1) = \chi(h_2) \neq 0$ and
$\chi(f)=1$}

The centralizer $\g_{\chi_s}$ is the even subalgebra
$\ev \g \cong \gl(2)$. By the identification via
$\rea{\chi}{\mathfrak{sl}(2)}$ as in
\ref{subsubsec:q(2)-cen-eq-nonzero},
we  show that
\[
\rea{\chi}{\g_{\chi_s}} \cong M_{p}(K)^{\oplus p} \oplus
M_{p}(K[x]/\langle x^2 \rangle)^{\oplus \frac{p(p-1)}{2}}.
\]
It follows by comparing with the algebra $\rea{\chi}{\g}$ computed in
Proposition~\ref{prop:q(2)-mixed} that $\rea{\chi}{\g}$ and
$\rea{\chi}{\g_{\chi_s}}$ are indeed Morita equivalent.

\begin{remark}
Recall \cite{WZ} the Super KW conjecture for a restricted Lie
superalgebra $\g$ states that the dimension of every
$\rea{\chi}{\g}$-module is divisible by $p^{\frac{d_0}{2}}2^{\lfloor
\frac{d_1}{2} \rfloor}$, where $\lfloor a \rfloor$ denotes the least
integer upper bound of $a$.

Let  $\g =\qn$ and let $\chi =\chi_s +\chi_n$ be the Jordan
decomposition of a general $p$-character $\chi$ for the Lie
superalgebra $\g$. Let $d_i = \dim \g_i - \dim \g_{\chi, i}$, $i \in
\Z_2$. By a comparison of dimensions using \eqref{even-cent} and
Lemma~\ref{lem:odd-cent}:
\begin{align*}
\sudim \g - \sudim \g_{\chi} = \sudim \g - \sudim
(\g_{\chi_s})_{\chi_n} = b_0|b_1 + (\sudim \g_{\chi_s} - \sudim
(\g_{\chi_s})_{\chi_n }),
\end{align*}
the Super KW  conjecture for $\g$ would follow from the validity of
Conjecture~\ref{conj:morita} when combined with the super KW
property for nilpotent $p$-character established in
Theorem~\ref{th:superKW}.

\end{remark}


\begin{thebibliography}{ABC}

\bibitem[BKN]{BKN} B. Boe, J. Kujawa and D. Nakano,
{\em Cohomology and support varieties for Lie superalgebras II}, Proc. London Math. Soc., {\bf 98} (2009), 19--44.

\bibitem[Br]{Br} J. Brundan,
{\em Kazhdan-Lusztig polynomials and character formulae for the Lie
superalgebra ${\mathfrak q}(n)$}, Adv. Math. {\bf 182} (2004),
28--77.

\bibitem[BK]{BK} J. Brundan and A. Kleshchev,
{\em Modular representations of the supergroup $Q(n)$, I}, J.
Algebra {\bf 260} (2003), 64--98.

\bibitem[CW]{CW} S.-J. Cheng and W. Wang,  {\em Remarks on the
Schur-Howe-Sergeev Duality}, Lett. Math. Phys. {\bf 52} (2000),
143--153.

\bibitem[FP]{FP} E. Friedlander and B. Parshall, {\em Modular
representation theory of Lie algebras}, Amer. J. Math. {\bf 110}
(1988), 1055--1093.

\bibitem[FM]{FM} A. Frisk and V. Mazorchuk, {\em Regular strongly typical blocks of
$\mathcal{O}^{\mathfrak{q}}$}, preprint (2008), arXiv:0810.4758.

\bibitem[Ger]{Ger} J. Germoni, {\em Indecomposable representations of special linear Lie
superalgebras}, J. Algebra {\bf 209} (1998), 367--401.

\bibitem[Gor]{Gor} I. Gordon, {\em Representations of semisimple Lie algebras
in positive characteristic and quantum groups at roots of unity},
London Math. Soc. Lect. Notes Series {\bf 290} (2001), 146--167.


\bibitem[HN]{HN} R. Holmes and D. Nakano,
{\em Brauer-type reciprocity for a class of graded associative
algebras}, J. Algebra {\bf 144} (1991), 117--126.

\bibitem[HJ]{HJ} R. Horn and C. Johnson,
{\em Topics in matrix analysis}, Cambridge University Press, Cambridge, 1994.

\bibitem[Jan]{Jan} J. Jantzen, {\em Representations of Lie algebras in prime
characteristic}, NATO Adv. Sci. Inst. Ser. C Math. Phys. Sci., {\bf
514}, Representation theories and algebraic geometry (Montreal, PQ,
1997), 185--235, Kluwer Acad. Publ., Dordrecht, 1998.

\bibitem[Pen]{Pen} I. Penkov, {\em Characters of typical irreducible
finite-dimensional $\mathfrak{q}(n)$-modules}, Funct. Anal. Appl.
{\bf 20} (1986), 30--37.



\bibitem[Pr1]{Pr1} A. Premet,
{\em Irreducible representations of Lie algebras of reductive groups
and the Kac-Weisfeiler conjecture}, Invent. Math. {\bf 121} (1995),
79--117.

\bibitem[Pr2]{Pr2}
A. Premet, {\em Special transverse slices and their enveloping
algebras}, Adv. Math. {\bf 170} (2002), 1--55.


\bibitem[Rud]{Rud} A. Rudakov,
{\em The representations of classical semisimple Lie algebras in
characteristic $p$},  Izv. Akad. Nauk SSSR Ser. Mat. {\bf 34}
(1970), 735--743.

\bibitem[Ser]{Ser} A. Sergeev,
{\em The tensor algebra of the identity representation as a module
over the Lie superalgebras $\mathfrak{gl}(n,m)$ and $Q(n)$}, Math.
USSR Sbornik {\bf 51} (1985), 419--427.


\bibitem[Skr]{Skr} S. Skryabin,
{\em Representations of the Poisson algebra in prime
characteristic}, Math. Z. {\bf 243} (2003), 563--597.

\bibitem[WK]{KW} B. Weisfeiler and V.~Kac,
{\em On irreducible representations of Lie $p$-algebras}, Func.
Anal. Appl. {\bf 5} (1971), 111--117.

\bibitem[WZ]{WZ} W. Wang and L. Zhao, {\em Representations of Lie superalgebras in prime characteristic
I}, Proc. London Math. Soc. {\bf 99} (2009), 145--167.


\bibitem[X]{Xia} J. Xiao, {\em Finite-dimensional representations of $U_t({\rm sl}(2))$ at roots of
unity},  Canad. J. Math. {\bf 49} (1997), 772--787.

\bibitem[Z]{Z} L. Zhao, {\em Finite $W$-superalgebras for queer Lie superalgebras and higher
Sergeev duality}, preprint, arXiv:1012.2326.

\end{thebibliography}
\end{document}